\documentclass[a4paper, 12pt]{amsart}
\usepackage{latexsym, amssymb, amscd, amsfonts, amsthm, amsmath, graphicx, mathabx, mathrsfs}

\newcommand{\Spec}{\operatorname{Spec}}

\newcommand{\supp}{\operatorname{supp}}
\newcommand{\ab}{\operatorname{ab}}

\newcommand{\good}{\operatorname{good}}

\newcommand{\stab}{\operatorname{stab}}
\newcommand{\Poly}{\operatorname{Poly}}

\newcommand{\IM}{\operatorname{Im}}
\newcommand{\RE}{\operatorname{Re}}

\newcommand{\bH}{\mathbb{H}}

\newcommand{\bR}{\mathbb{R}}
\newcommand{\bU}{\mathbb{U}}
\newcommand{\zed}{\mathbb{Z}}

\newcommand{\be}{\mathbf{e}}

\newcommand{\id}{\mathrm{id}}
\newcommand{\SO}{\mathrm{SO}}

\newcommand{\one}{\mathbf{1}}
\newcommand{\E}{\mathbf{E}}

\newcommand{\sA}{\mathscr{A}}
\newcommand{\sF}{\mathscr{F}}

\newcommand{\sT}{\mathscr{T}}

\newcommand{\ve}{{\varepsilon}}

\newcommand{\ui}{\underline{i}}
\newcommand{\uv}{\underline{v}}

\newcommand{\hw}{\hat{w}}
\newcommand{\ha}{\hat{\alpha}}

\newcommand{\uw}{\underline{w}}
\newcommand{\huw}{\hat{\underline{w}}}
\newcommand{\ouw}{\overline{\underline{w}}}

\newcommand{\utau}{\underline{\tau}}

\newcommand{\um}{\underline{m}}
\newcommand{\un}{\underline{n}}
\newcommand{\up}{\underline{p}}

\newcommand{\ux}{\underline{x}}

\newcommand{\uy}{\underline{y}}

\newcommand{\oux}{\overline{\underline{x}}}
\newcommand{\ouy}{\overline{\underline{y}}}

\newcommand{\sN}{{\mathscr{N}}}
\newcommand{\sM}{{\mathscr{M}}}
\newcommand{\sI}{{\mathscr{I}}}
\newcommand{\sE}{{\mathscr{E}}}
\newcommand{\sO}{{\mathscr{O}}}
\newcommand{\fS}{{\mathfrak{S}}}

\newtheorem{theorem}{Theorem}

\newtheorem{lemma}[theorem]{Lemma}
\newtheorem{proposition}[theorem]{Proposition}

\theoremstyle{remark}
\newtheorem*{rem}{Remark}

\title{Random walk on unipotent matrix groups}
\author{Persi Diaconis}
\address[Persi Diaconis]{Department of Mathematics, Stanford University, 450 Serra Mall,
Stanford, CA, 94305, USA}
\email{diaconis@math.stanford.edu}

\author{Robert Hough}
\address[Robert Hough]{Department of Mathematics, Stanford University, 450 Serra Mall,
Stanford, CA, 94305, USA}
\curraddr{School of Mathematics, Institute of Advanced Study, 1 Einstein Drive, Princeton, NJ, 08540}
\email{hough@math.ias.edu}


\thanks{We are grateful to Laurent Saloff-Coste, who provided us a detailed
account of previous work. We thank the referee for a careful reading of the manuscript.
}

\thanks{This material is based upon work supported by the National Science Foundation under agreements No.
DMS-1128155 and DMS-1712682. Any opinions, findings and conclusions or recommendations expressed
in this material are those of the authors and do not necessarily reflect the views of the National Science
Foundation.
}

\begin{document}

\begin{abstract}
We introduce a new method for proving central limit theorems for random walk on
nilpotent groups.  The method is illustrated in a local central limit theorem
on the Heisenberg group, weakening the necessary conditions on the driving
measure. As a second illustration, the method is used to study  walks
on the $n\times n$ uni-upper triangular group with entries taken modulo $p$.
The method allows sharp answers to the behavior of individual coordinates:
coordinates immediately above the diagonal require order $p^2$ steps for
randomness, coordinates on the second diagonal require order $p$ steps;
coordinates on the $k$th diagonal require order $p^{\frac{2}{k}}$ steps.
\end{abstract}

\maketitle

\section{Introduction}
Let $\bH(\bR)$ denote the real Heisenberg group
\begin{equation}
 \bH(\bR) = \left\{\left( \begin{array}{ccc} 1 & x &z\\ 0 & 1&y \\0&0&1 \end{array}\right): x, y, z \in \bR\right\}.
\end{equation}
Abbreviate $\left( \begin{array}{ccc} 1 & x &z\\ 0 & 1&y \\0&0&1
\end{array}\right)$ with $[x,y,z]$, identified with a vector in $\bR^3$.
Consider simple random walk on $G=\bH(\bR)$ driven by Borel probability measure
$\mu$.  For $N \geq 1$, the law of this walk is the convolution power $\mu^{*N}$
where, for Borel measures $\mu, \nu$ on $G$, and for $f \in C_c(G)$,
\begin{equation}
 \langle f, \mu *\nu \rangle = \int_{g, h \in G} f(gh)\; d\mu(g)\; d\nu(h).
\end{equation}

Say that measure $\mu$ is non-lattice (aperiodic) if its support is not
contained in a proper closed subgroup of $G$.  For general non-lattice $\mu$ of
compact support
Breuillard \cite{B05} uses the representation theory of $G$ to prove a local
limit theorem for the law of $\mu^{*N}$, asymptotically evaluating its density
in translates of bounded Borel sets.  However, in evaluating $\mu^{*N}$
on Borel sets translated on both the left and the right he makes a decay
assumption on the Fourier transform of the abelianization of the measure $\mu$,
and raises the question of whether this is needed.
We show that this condition is unnecessary.  In doing so we give an alternative
approach to the local limit theorem on
$G$ treating it  as an extension of the classical
local limit theorem on $\bR^n$. 
We also 
obtain the best possible rate. The
method of argument is analogous to (though simpler than) the analysis of
quantitative equidistribution of polynomial orbits on $G$ from \cite{GT12}.

Recall that the abelianization   $G_{\ab} = G/[G,G]$ of $G$ is isomorphic to
$\bR^2$ with projection $p:G \to G_{\ab}$ given by $p([x,y,z]) = [x,y]$. Assume
that the probability measure $\mu$  satisfies the following
conditions.
\begin{enumerate}
 \item[i.] \emph{Compact support. }
 \item[ii.] \emph{Centered. } The projection $p$ satisfies
  \begin{equation}
  \int_{G} p(g) d\mu(g) = 0.
 \end{equation}
 \item[iii.] \emph{Full dimension. } Let $\Gamma = \overline{\langle \supp \mu  
\rangle}$ be the closure of the subgroup of $G$ generated by the support of 
$\mu$.  The quotient $G/\Gamma$ is compact.
\end{enumerate}
Section \ref{background_section} gives a characterization of closed subgroups
$\Gamma$ of $G$ of full dimension.

Under the above conditions, the central limit theorem for $\mu$ is known. 
 Let $(d_t)_{t > 0}$ denote the
semigroup of dilations given by 
\begin{equation}d_t([x,y,z]) = [tx, ty, t^2 z]\end{equation} and denote the
Gaussian semigroup $(\nu_t)_{t > 0}$ defined by its generator (see
\cite{B05}, \cite{VSC92})
\begin{align}
 \sA f &= \frac{d}{dt}\bigg|_{t = 0} \int_{g \in G} f(g) d\nu_t(g)\\
 \notag  &= \overline{z} \partial_z f(\id) + \overline{xy}\partial^2_{xy}f(\id) + \frac{1}{2}\overline{x^2}\partial_x^2 f(\id) + \frac{1}{2}\overline{y^2}\partial_y^2 f(\id)
\end{align}
where $\sigma_x^2 = \overline{x^2}= \int_{g=[x,y,z] \in G} x^2 d\mu(g)$ and similarly $\sigma_y^2 = \overline{y^2}$, $\sigma^2_{xy}=\overline{xy}$, $\overline{z}$.   With $\nu = \nu_1$, the central limit theorem for $\mu$ states that for $f \in C_c(G)$,
\begin{equation}
 \left\langle f, d_{\frac{1}{\sqrt{N}}}\mu^{*N}\right\rangle \to \langle f, \nu \rangle.
\end{equation}
For $g \in G$   define the left and right translation operators $L_g, R_g:
L^2(G) \to L^2(G)$,
\begin{equation}
 L_g f(h) = f(gh), \qquad R_g f(h) = f(hg).
\end{equation}
Our local limit theorem in the non-lattice case is as follows.

\begin{theorem}\label{local_limit_theorem}
Let $\mu$ be a Borel probability measure of compact support on $G = \bH(\bR)$, which is centered and full dimension. Assume that the projection to the abelianization $\mu_{\ab}$ is non-lattice.  
Let $\nu$ be the limiting Gaussian measure of $d_{\frac{1}{\sqrt{N}}}\mu^{*N}$.
For $f \in C_c(G)$, uniformly for $g, h \in G$, as  $N \to \infty$,
\begin{equation}\label{weak_local_limit}
 \left\langle L_g R_h f,  \mu^{*N}  \right\rangle = \left\langle L_g R_h f, d_{\sqrt{N}}\nu \right\rangle + o_{\mu,f}\left(  N^{-2} \right).
\end{equation}
% Updated 10/18
If the Cram\'{e}r condition holds:
 \begin{equation}
 \sup_{\lambda \in \widehat{\bR^2}, \; |\lambda| > 1} \left|\int_{g=[x,y,z] \in G} e^{-i\lambda \cdot (x,y)} d\mu(g)\right| <1
\end{equation}
then uniformly for $g, h \in G$ and Lipschitz $f \in C_c(G)$, as $N \to \infty$
\begin{equation}\label{strong_local_limit}
 \left\langle L_g R_h f, \mu^{*N}  \right\rangle = \left\langle L_g
R_h f, d_{\sqrt{N}}\nu \right\rangle + O_{f,\mu}\left(
N^{-\frac{5}{2}} \right).
\end{equation}

\end{theorem}

\begin{rem}
The rate is best possible as may be seen by projecting to the abelianization.
 A variety of other statements of the local theorem are also derived, see eqn.
(\ref{lattice_target}) in Section \ref{summary_of_argument_section}.
\end{rem}

\begin{rem}
 For non-lattice $\mu$, \cite{B05} obtains (\ref{weak_local_limit}) with $h =
\id$ and for general $h$ subject to Cram\'{e}r's condition.  A condition
somewhat weaker than Cram\'{e}r's would suffice to obtain
(\ref{strong_local_limit}).
\end{rem}

\begin{rem}
 In the case that $\mu$ is supported on a closed discrete subgroup or has a density with respect to Haar measure, \cite{A02a, A02b} obtains
an error of $O\left(N^{-\frac{5}{2}}\right)$  in
approximating $\mu^{*N}(g)$, $g \in \Gamma$.
\end{rem}

Our proof of Theorem \ref{local_limit_theorem} applies equally well in the case when $\mu_{\ab}$ has a lattice component, and gives a treatment which is more explicit than the argument in \cite{A02a}.  To illustrate this, we determine the leading constant in the probability of return to 0 in simple random walk on $\bH(\zed)$. 
\begin{theorem}\label{zero_return_theorem}
 Let $\mu_0$ be the measure on $\bH(\zed)$ which assigns equal probability
$\frac{1}{5}$ to each element of the generating set \begin{equation}\left\{\id,
\left(\begin{array}{ccc} 1 & \pm 1 &0 \\ 0 & 1 & 0\\ 0&0&1\end{array} \right),
\left(\begin{array}{ccc}1 &0 &0\\ 0 & 1 & \pm 1\\ 0&0&1
\end{array} \right)\right\}.\end{equation}
As $N \to \infty$,
$
\mu_0^{*N}(\id) = \frac{25}{16 N^2} + O\left(N^{-\frac{5}{2}}\right).
$
\end{theorem}

The basic idea which drives the  proof is that permuting segments of generators
in a typical word of the walk generates smoothness in the central coordinate of
the product, while leaving the abelianized coordinates unchanged.  This
observation permits passing from a limit theorem to a local limit theorem by
smoothing at a decreasing sequence of scales. 
When studying $\mu^{*N}$ near the scale of its distribution, we use a Lindeberg replacement scheme in which one copy at a time of $\mu$ is replaced with a Gaussian measure in the abelianization. To handle uniformity in the translation in Theorem \ref{local_limit_theorem} in the case where the Cram\'{e}r condition is not assumed we are forced to treat frequencies $\alpha$ which are unbounded, and thus must consider the large spectrum 
\begin{equation}
\Spec_\vartheta(\mu_{\ab}) = \left\{\alpha \in \bR^2 : \left|\hat{\mu}_{\ab}(\alpha)\right|> 1-\vartheta \right\}
\end{equation} where $\vartheta \to 0$ as a function of $N$.  In treating this, we use an approximate lattice structure of $\Spec_\vartheta(\mu_{\ab})$, see Section \ref{large_spectrum_section}.

 As a further application of the word rearrangement technique, answering a question of \cite{DS94} we determine the mixing time
of the central coordinate in a natural class of random walks on  the group
$N_n(\zed/p\zed)$ of $n\times n$ uni-upper triangular matrices with entries in
$\zed/p\zed$.

\begin{theorem}\label{N_n_theorem}
 Let $n \geq 2$ and let $\mu$ be a probability measure on $\zed^{n-1}$ which
satisfies the following conditions.
 \begin{enumerate}
  \item[i.] \emph{Bounded support.}
  \item[ii.] \emph{Full support.} $\langle \supp \mu \rangle = \zed^{n-1}$
  \item[iii.] \emph{Lazy.} $\mu(0) > 0$
  \item[iv.] \emph{Mean zero.} $\sum_{x \in \zed^{n-1}} x \mu(x) = 0$
  \item[v.] \emph{Trivial covariance.}
  \begin{equation}
   \left(\sum_{x \in \zed^{n-1}} x^{(i)}x^{(j)} \mu(x) \right)_{i,j=1}^{n-1} = 
I_{n-1}.
  \end{equation}
 \end{enumerate} Push forward $\mu$ to a probability measure $\tilde{\mu}$ on $N_n(\zed)$ via, for all $x \in \zed^{n-1}$,
 \begin{equation}
  \tilde{\mu}\left(\begin{array}{ccccc}
1 & x^{(1)} &0 &\cdots &0\\
0&  1 &     x^{(2)}&\ddots&\vdots\\
\vdots & \ddots   & \ddots & \ddots&0 \\
 &    &       0  & 1 & x^{(n-1)}\\
 0 &\cdots  &          & 0 & 1
 \end{array}
 \right) = \mu(x).
 \end{equation}
 Write $Z: N_n(\zed)\to \zed$ for the upper right corner entry of a matrix of
$N_n(\zed)$.  There exists $C >0$  such that, for all primes $p$, for $N \geq
1$,
\begin{align}
\sum_{x \bmod p} \left| \tilde{\mu}^{*N}(Z \equiv x \bmod p) -
\frac{1}{p}\right|  \ll \exp\left(-C\frac{N}{p^{\frac{2}{n-1}}}\right).
\end{align}
 \end{theorem}

\begin{rem}
 Informally, the top right corner entry mixes in time
$O\left(p^{\frac{2}{n-1}}\right)$.  This is tight, since archimedean
considerations show that the $L^1$ distance to uniform is $\gg 1$ if the number
of steps of the walk is $\ll p^{\frac{2}{n-1}}$.
\end{rem}

 \begin{rem}
  Although we have considered only the top right corner entry in 
$U_n(\zed/p\zed)$, this result  determines the mixing time of each entry above 
the diagonal by iteratively projecting to the subgroups  determined by the top 
left or bottom right $m\times m$ sub-matrices.
 \end{rem}

\begin{rem}
 Our argument permits treating measures not supported on the first super-diagonal essentially without change, since entries above the first diagonal introduce a lower degree tensor which is annihilated by the application of the Gowers-Cauchy-Schwarz inequality.  We treat the simplified case stated in order to ease the notation.
\end{rem}

\section*{History}
Random walk on groups is a mature subject with myriad projections into
probability, analysis and applications.  Useful overviews with extensive
references are in \cite{B04}, \cite{S01}.  Central limit theorems for random
walk on Lie groups were first proved by \cite{W62} with \cite{T64} carrying out
the details for the Heisenberg group.  Best possible results under a second
moment condition for nilpotent Lie groups are in \cite{R78}.

A general local limit theorem for the Heisenberg group appears in \cite{B05},
which contains a useful historical review.  There similar conditions to those of
our Theorem \ref{local_limit_theorem} are made, but the argument treats only the
non-lattice
case and needs a stronger condition on the characteristic function of the
measure projected to the abelianization.  Remarkable local limit theorems are
in \cite{A02a, A02b}.  The setting is groups of polynomial growth, and so
``essentially'' nilpotent Lie groups via Gromov's Theorem.  The first paper
gives quite complete results assuming that the generating measure has a
density.  The second paper gives results for measures supported on a lattice.
The arguments in \cite{A02b} have been adapted in \cite{B10} to give a local
limit theorem for non-lattice measures supported on finitely many points.

Just as for the classical abelian case, many variations have been studied.
Central limit theorems for walks satisfying a Lindeberg condition on general
Lie groups are proved in \cite{SV73}, see also references therein. Large
deviations for walks on nilpotent groups are proved in \cite{BC99}.
Central limit theorems on covering graphs with nilpotent automorphism groups
are treated in \cite{I03, I04}.  This allows walks on Cayley graphs with some
edges and vertices added and deleted.  Brownian motion and heat kernel
estimates are also relevant, see \cite{H76, G77}.

Random walk on finite nilpotent groups are a more recent object of study.
Diaconis and Saloff-Coste \cite{DS96, DS95, DS94} show that for simple
symmetric random walk on $\zed/n\zed$, order $n^2$ steps are necessary and
sufficient for convergence to uniform.  The first paper uses Nash inequalities,
the second lifts to random walk on the free nilpotent group and applies central
limit theorems of Hebisch, Saloff-Coste and finally Harnack inequalities to
transfer back to the finite setting.  The third paper uses geometric ideas of
moderate growth to show that for groups of polynomial growth, diameter-squared
steps are necessary and sufficient to reach uniformity.  This paper raises the
question of the behavior of the individual coordinates on $U_n(\zed/p\zed)$
which is finally answered in Theorem \ref{N_n_theorem}.  A direct non-commuting
Fourier approach to $\bH(\zed/p\zed)$ is carried out in \cite{BDHMW15}, where 
it is shown that order $p\log p$ steps suffice to make the central 
coordinate random,
improved here to order $p$ steps, which is best possible. For a review of the
$\bH(\zed)$ results, see \cite{D10}. Finally there have been quite a number
of papers studying the walk on $U_n(\zed/p\zed)$ when both $p$ and $n$ grow.
We refer to \cite{PS13}, which contains a careful review and definitive results.

 \section*{Notation and conventions}
 Vectors from $\bR^d$, $d \geq 1$ are written in plain text $w$, their
coordinates with superscripts  $w^{(i)}$, and sequences of vectors with an
underline $\uw$.  The sum of a sequence of vectors $\uw$ is indicated $\ouw$.
$w^t$ denotes the transpose of $w$.  We have been cavalier in our use of 
transpose; interpretation of vectors as rows or columns should be clear from 
the context.  We frequently identify matrix elements in
the group $U_n$ with vectors from Euclidean space, and have attempted to
indicate the way in which the vectors should be interpreted.  As a rule of
thumb, when the group law is written multiplicatively, the product is in the
group $U_n$, and when additively, in Euclidean space.

The arguments presented use permutation group actions on sequences of vectors.
Given integer $N \geq 1$, denote $\fS_N$ the symmetric group on $[N] = \zed
\cap [1,N]$, which acts on length $N$ sequence of vectors by permuting the
indices:
\begin{equation}
\fS_N \ni \sigma: (w_1, ..., w_N) \mapsto (w_{\sigma(1)}, ..., w_{\sigma(N)}).
\end{equation}
$C_2$ is the two-element group.  For $d \geq 1$, identify $C_2^d$ with the
$d$-dimensional hypercube $\{0,1\}^d$.  $\one_d$ is the element of $C_2^d$
corresponding to the sequence of all 1's on the hypercube.  $C_2^d$ acts on
sequences of vectors of length $2^{d}$ with the $j$th factor determining the
relative order of the first and second blocks of $2^{j-1}$ elements.   To
illustrate the action of $C_2^2$ on $\ux = x_1x_2x_3x_4$:
\begin{align}
\notag (0,0) \cdot \ux &= x_1x_2x_3x_4\\
 (1,0) \cdot \ux &= x_2x_1x_3x_4 \\
\notag (0,1) \cdot \ux &= x_3x_4x_1x_2\\
\notag (1,1) \cdot \ux &= x_3x_4x_2x_1.
\end{align}

The 2-norm on $\bR^d$ is indicated $\|\cdot \|$ and $\|\cdot\|_{(\bR/\zed)^d}$ 
denotes distance to the nearest integer lattice point.  
Given $\xi \in \bR^d$, $e_\xi(\cdot)$ denotes the character of $\bR^d$,
$e_\xi(x) = e^{2\pi i \xi \cdot x}$.  

Use $\delta_x$ to indicate the Dirac delta measure at $x \in \bR^d$.
 Given $f \in C_c(\bR^d)$ and measure $\mu$, $\langle f, \mu\rangle$ denotes the 
bilinear pairing
 \begin{equation}
  \langle f, \mu \rangle = \int_{\bR^d} f(x)d\mu(x).
 \end{equation}
Denote the Fourier transform of function $f$, resp. the characteristic function
of measure $\mu$ by, for $\xi \in \bR^d$,
\begin{equation}
\hat{f}(\xi) = \int_{\bR^d} e_{-\xi}(x) f(x)dx, \qquad 
\hat{\mu}(\xi) =\int_{\bR^d} e_{-\xi}(x) d\mu(x).
\end{equation}
For $x \in \bR^d$,
$T_xf$ denotes function $f$ translated by $x$,
\begin{equation}
 T_x f(y) = f(y-x), \qquad \widehat{T_x f}(\xi) = e_{-\xi}(x) 
\hat{f}(\xi)
\end{equation}
and for real $t>0$, $f_t$ denotes $f$ dilated by $t$,
\begin{equation}
 f_t(x) = t^d f\left(tx\right), \qquad \widehat{f_t}(\xi) = 
\hat{f}\left(\frac{\xi}{t}\right).
\end{equation}
For smooth $f$, 
\begin{equation}
f(x) = \int_{\bR^d}\hat{f}(\xi)e_{\xi}(x)d\xi.
\end{equation}

By a bump function $\rho$ on $\bR^n$ we mean a smooth non-negative function of compact support and integral 1.  The Fourier transform of $\rho$ satisfies, for each $A > 0$ there is a constant $C(A, \rho) > 0$ such that 
\begin{equation}
 |\hat{\rho}(\xi)| \leq \frac{C(A, \rho)}{(1 + \|\xi\|)^A}.
\end{equation}
This follows from integration by parts.

 For $r \in \bR$ and $\sigma > 0$, $\eta(r, \sigma)$ denotes the one-dimensional
Gaussian distribution with mean $r$ and variance $\sigma^2$, with density and
characteristic function
 \begin{equation}
  \eta(r, \sigma)(x) = 
\frac{\exp\left(-\frac{(x-r)^2}{2\sigma^2}\right)}{\sqrt{2\pi} \sigma}
, \qquad \widehat{\eta(r,\sigma)}(\xi) = 
e_{-\xi}(r) \exp\left( -2\pi^2 \sigma^2 \xi^2\right).
 \end{equation}
A centered (mean zero) normal distribution $\eta$ in dimension $d$ is specified
by its covariance matrix
\begin{equation}
 \sigma^2 = \left(\int_{\bR^d} x^{(m)} x^{(n)} \eta(x) \right)_{m,n = 1}^d
\end{equation}
and has density and characteristic function
\begin{equation}
 \eta(0, \sigma)(x) = \frac{\exp\left(-\frac{ x^t (\sigma^2)^{-1}
x}{2}\right)}{(2\pi)^{\frac{d}{2}} \left(\det \sigma^2\right)^{\frac{1}{2}}} , 
\qquad
\widehat{\eta(0, \sigma)}(\xi) = \exp\left(-2\pi^2 \xi^t \sigma^2 \xi\right).
\end{equation}

All of our arguments concern the repeated convolution $\mu^{*N}$ of a fixed
measure $\mu$ on the upper triangular matrices. The product measure $\mu^{\otimes N}$ is abbreviated $\bU_N$. Asymptotic statements are with
respect to $N$ as the large parameter.
The Vinogradov notation $A \ll B$, resp. $A \gg B$, means $A =
O(B)$, resp. $B = O(A)$. $A \asymp B$ means $A \ll B$ and $B \ll A$.

\section{Background to Theorems
\ref{local_limit_theorem} and \ref{zero_return_theorem}}\label{background_section}
This section collects together several background statements regarding the
Heisenberg group, its Gaussian semigroups of probability measures and statements
of elementary probability which are needed in the course of the argument.

Write $A = [1, 0, 0]$, $B = [0,1,0]$, $C = [0, 0, 1]$.  The following
commutators are useful,
\begin{align}
\notag [A, B] &=ABA^{-1}B^{-1}= [0,0,1] = C,
\\ [A^{-1}, B^{-1}] &=A^{-1}B^{-1}AB= [0,0,1] = C,
\\\notag [A, B^{-1}] &=AB^{-1}A^{-1}B= [0,0,-1] = C^{-1},
\\\notag  [A^{-1}, B] &= A^{-1}BAB^{-1} = [0,0,-1] = C^{-1}.
\end{align}
A convenient representation for $[x,y,z] \in \bH(\bR)$ is $C^zB^yA^x$.  Using
the commutator rules above, the multiplication rule
for $\uw \in \bH(\bR)^N$ becomes
\begin{equation}\label{mult_rule}
 \prod_{i=1}^N \left[w_i^{(1)}, w_i^{(2)}, w_i^{(3)}\right] = \left[\ouw^{(1)}, 
\ouw^{(2)}, \ouw^{(3)} + H(\uw) \right]
\end{equation}
where $\overline{\cdot }$ and $H$ act on sequences of vectors from $\bR^3$ via
\begin{equation}\label{def_overline}\ouw = \sum_i w_i \qquad H(\uw) = \sum_{i < j} w_i^{(1)} w_j^{(2)}.
\end{equation}
It is also convenient to define
\begin{align}\label{def_H_ast}
 H^*(\uw) &= H(\uw) - \frac{1}{2}\ouw^{(1)}\ouw^{(2)} + \frac{1}{2}\sum_{i=1}^N 
w_i^{(1)}w_i^{(2)} \\\notag &= \frac{1}{2}\sum_{1 \leq i < j \leq 
N}\left(w_i^{(1)}w_j^{(2)} - w_i^{(2)}w_j^{(1)} \right),
\end{align}
and for $w = [x,y,z]$, $\tilde{w} = \left[x,y, z-\frac{1}{2} xy\right]$, so that 
the multiplication rule may be written
\begin{equation}\label{multiplication_rule}
 \prod_{i=1}^N w_i = \overline{\underline{\tilde{w}}} + \left[0,0, \frac{1}{2} 
\ouw^{(1)}\ouw^{(2)} + H^*(\uw)\right].
\end{equation}

Let $S = \supp \mu$.  Recall that $\Gamma = \overline{\langle S \rangle}$ is the
closure of the group generated by $S$.  Its abelianization, $\Gamma_{\ab} =
\Gamma/[\Gamma, \Gamma]$ is equal to $p(\Gamma)$ where $p$ is the projection $p:
G \to G_{\ab}$.  Let $\Gamma_0$ be the semigroup generated by $S$.  We record
the following descriptions of $\Gamma$ and $\Gamma_0$.

\begin{proposition}\label{closed_subgroup_description}
Let $\Gamma \leq \bH(\bR)$ be a closed subgroup of full dimension. The structure 
of the abelianization $\Gamma_{\ab} = \Gamma/[\Gamma, \Gamma]$ and of $\Gamma$ 
falls into one of the following cases.
 \begin{enumerate}
  \item[i.]
  \begin{equation}\Gamma_{\ab} = \bR^2,  \qquad \Gamma = \left\{[\gamma, r]: \gamma \in 
\Gamma_{\ab}, r \in \bR\right\}\end{equation}
  \item[ii.] There exist non-zero orthogonal $v_1, v_2 \in \bR^2$, such that
    \begin{align}&\Gamma_{\ab} = \{n v_1 + rv_2: n \in \zed, r \in \bR\}, \\ \notag & \Gamma = 
\left\{[\gamma, r]: \gamma \in \Gamma_{\ab}, r \in \bR \right\} \end{align}
  \item[iii.] There exist non-zero $v_1, v_2 \in \bR^2$, linearly independent over 
$\bR$, such that \begin{equation}\Gamma_{\ab} = \{n_1 v_1 + n_2 v_2: n_1, n_2 \in \zed\}.\end{equation} In 
this case, $\Gamma$ falls into one of two further cases
  \begin{enumerate}
    \item[iv.]  $\Gamma = \left\{[\gamma, r]: \gamma \in \Gamma_{\ab}, r \in \bR
\right\}$
    \item[v.] There exists $\lambda \in \bR_{>0}$ and $f: \Gamma_{\ab}\to \bR$ such that
    \begin{equation}
     \Gamma = \left\{[\gamma, \lambda(f(\gamma) + n)]: \gamma \in \Gamma_{\ab}, 
n \in \zed\right\}.
    \end{equation}

  \end{enumerate}
 \end{enumerate}
\end{proposition}

\begin{proof}[Proof of Proposition \ref{closed_subgroup_description}]
 The full dimension condition guarantees that $\Gamma_{\ab}$ is a two
dimensional closed subgroup of $\bR^2$, and the three possibilities given are
all such closed subgroups. Likewise, the center of $\Gamma$ is a non-trivial
subgroup of $\bR$, hence either $\bR$ or $\lambda \cdot \zed$ for some real
$\lambda > 0$.  It follows that the fiber over $\gamma \in \Gamma_{\ab}$ is a
translate of the center. Let $v_1, v_2$ be two linearly independent elements of
the abelianization, and choose $g_1 = [v_1, z_1]$, $g_2 = [v_2, z_2]$ in
$\Gamma$. The commutator $[g_1, g_2] = g_1g_2g_1^{-1}g_2^{-1}$ is bilinear in
$v_1$, $v_2$, is non-zero, and lies in the center.  It follows that if one of
$v_1, v_2$ may be scaled by a continuous parameter in the abelianization then
the center is $\bR$.
\end{proof}

\begin{lemma}
 The closure of the semigroup $\Gamma_0$ is $\overline{\Gamma_0} = \Gamma$.
\end{lemma}

\begin{proof}
 Write $\Gamma_{0, \ab} = p(\Gamma_0)$ where $p$ denotes projection to the
abelianization $G_{\ab}$. That $\overline{\Gamma_{0,\ab}} =
\Gamma_{\ab}$ follows from the local limit theorem on $\bR^2$.  To treat the central fiber, in the case $\Gamma_{\ab} = \bR^2$ let $0 < \epsilon < \frac{1}{4}$ be a fixed small parameter and choose $x, x',
y, y'$ in $\Gamma_0$ such that 
\begin{equation}
p(x), p(x'), p(y), p(y') \approx e_1, -e_1,
e_2, -e_2
\end{equation} where the approximation means to within distance $\epsilon$.
Take a word $\uw$ in $T= \{\id, x,x', y,y'\}$ of length 
$4n$ with product
approximating the identity in $\Gamma_{\ab}$ to within $\epsilon$, which is such
that each of $x, x', y, y'$ appear $>(1-O(\epsilon)) n$ times in $\uw$.  The
abelianization of the product is independent of the ordering of $\uw$, but  if
the letters in $\uw$ appear in order $y, x, y', x'$ then   the central element
is $<  -(1 + O(\epsilon))n^2$, while if they appear in order $y', x, y, x'$ then
the central element is $>(1 + O(\epsilon)) n^2$.  Moving from an ordering of the
first type to an ordering of the second by swapping generators one at a time
changes the central element by $O(1)$ at each step.  Let $\epsilon \downarrow
0$ to deduce that $\overline{\Gamma_0}$ contains positive and negative central
elements, and hence that $\overline{\Gamma_0}$ is a group, equal to $\Gamma$. In the case $\Gamma_{\ab}$ has a one or two dimensional lattice component, replace either $e_1$ or both $e_1, e_2$ above with a basis for the lattice component and repeat the argument.
\end{proof}

More quantitative structural statements are as follows.

\begin{lemma}\label{H_star_cramer}
 Let $\mu$ be a measure on $\bH(\bR)$, with abelianization $\mu_{\ab}$ not
supported on a lattice of $\bR^2$. If the Cram\'{e}r condition holds for the
measure $\mu_{\ab}$ then it holds also for the measure  $\mu'$ on $\bR$ 
obtained
by
pushing forward $\mu_{\ab}\otimes \mu_{\ab}$ by $H^*(w_1, w_2)$.
\end{lemma}
\begin{proof}
Let $\xi \in \bR$, $|\xi|\geq 1$ and fix $w_2\in \supp(\mu_{\ab})$, bounded
away
from 0.  Write $H^*(w_1, w_2) = \frac{w_1 \wedge w_2}{2} = \frac{1}{2}w_1 \cdot
\hat{w}_2$.  The claim follows since $\left|\int e_{-\xi}\left(H^*(w_1, 
w_2)\right) d
\mu_{\ab}(w_1)\right|$ is bounded away from 1 uniformly in $\xi$ and $w_2$.
\end{proof}

\begin{lemma}\label{H_star_decay}
 Let $\mu$ be a measure on $\bR^2$ of compact support, with support generating
a subgroup of $\bR^2$ of full dimension.  If $\mu$ is lattice supported, assume
that the co-volume of the lattice is at least 1.  There is a constant $c =
c(\mu) > 0$ such that, uniformly in $0 < \xi \leq \frac{1}{2}$, for $N = N(\xi) =
\left \lfloor \frac{1}{2\xi} \right \rfloor$,
\begin{equation}
\left| \int_{\bR^2 \times \bR^2}e_{-\xi}\left( H^*(w_1, w_2)\right)
d\mu^{*N}(w_1)d\mu^{*N}(w_2)\right| \leq 1- c(\mu).
\end{equation}
\end{lemma}
\begin{proof}

When $\mu$ is lattice with lattice of covolume $V$, the measure $H^*(w_1,w_2) d\mu(w_1)d\mu(w_2)$ is lattice distributed with step size $V$.  Hence the bound on $|\xi|$ suffices to guarantee the claim for $N$ bounded.  

For $N$ growing, a standard application of the functional central limit theorem implies that
$\frac{1}{N} H^*(w_1, w_2) d\mu^{*N}(w_1)d\mu^{*N}(w_2)$ converges to a
non-zero density on $\bR$ as $N \to \infty$.
\end{proof}

Normalize Haar measure on $\bH(\bR)$ to be given in coordinates by $dg =dx dy 
dz$. 
The density of a Gaussian measure $\nu$ on $\bH(\bR)$ can be understood as
the rescaled limit of the density of a random walk with
independent Gaussian inputs in the abelianization.  Consider the distribution on
the Heisenberg group given by $\nu_{2,\sigma} = [\eta(0,\sigma),  0]$, which has 
projection to the abelianization given by a two dimensional 
normal distribution of covariance $\sigma$, and with trivial central fiber. Write $\nu_2 = \nu_{2, I_2}$ for the measure in which $\sigma$ is the two dimensional identity matrix. The rescaled distribution
$d_{\frac{1}{\sqrt{N}}} \nu_2^{*N}$ converges to a Gaussian measure $\nu_0$ on
$\bH(\bR)$ as $N\to \infty$. Note that we have not included a covariance term,
which can be accommodated with a linear change of coordinates.  Also, we do not
consider randomness in the central coordinate as it would scale only as
$\sqrt{N}$, whereas the central coordinate has distribution at scale $N$.

Let $\alpha \in \bR^2$ and $\xi \in \bR$.  Write the modified characteristic 
function of $\nu_0$ as (recall $\tilde{z} = z - \frac{xy}{2}$)
\begin{equation}
 I(\alpha,  \xi) = \int_{g = [x,y,z] \in G} 
e_{-\alpha}(g_{\ab}) e_{-\xi}(\tilde{z}) d\nu_0(g)
\end{equation}
and
\begin{align}
 &I(\alpha, \xi; N) = \int_{(\bR^{2})^N} e_{-\alpha}\left(\frac{\oux}{\sqrt{N}} 
\right)e_{-\xi}\left(\frac{H^*(\ux)}{N} \right)  
d\nu_{2,\ab}^{\otimes N}\left(\ux\right).
\end{align}
\begin{lemma}
 Let $\alpha \in \bR^2, \xi \in \bR$ and let $\sigma^2$ be the covariance matrix of a non-degenerate two dimensional normal distribution of determinant $\delta^2 = \det \sigma^2$, $\delta>0$.  Then
 \begin{align}
  & \int_{(\bR^{2})^N} e_{-\alpha}\left(\frac{\oux}{\sqrt{N}} 
\right)e_{-\xi}\left(\frac{H^*(\ux)}{N} \right)  
d\eta(0,\sigma)^{\otimes N}\left(\ux\right) = I(\sigma \alpha, \delta \xi; N).
 \end{align}

\end{lemma}

\begin{proof}
 Making the change of variables, for each $i$, $\sigma y_i = x_i$ in the density 
$\frac{1}{2\pi \delta} \exp\left(- \frac{x_i^t \sigma^{-2} x_i}{2}\right)$ changes $\oux = \sigma \ouy$ and $H^*(\ux) = \det \sigma  \cdot H^*(\uy)$.
\end{proof}

In view of the multiplication rule (\ref{multiplication_rule}), for 
$\|\alpha\|, |\xi| = O(1)$
\begin{equation}
 \lim_{N \to \infty}  I(\alpha,  \xi; N) \to I(\alpha,  \xi).
\end{equation}
The following rate of convergence is given in Appendix \ref{char_fun_section}.
\begin{theorem}\label{char_fun_theorem}
 For all $\alpha \in \bR^2$, $\xi \in \bR$ such that $(1 + \|\alpha\|^2)(1 +
\xi^2)< N$,
 \begin{equation}
  I\left(\alpha, \xi; N\right) = \frac{1 + O\left( \frac{( 1+\|\alpha\|^2
)(1 + \xi^2)}{N} \right)}{\exp\left(\frac{2\pi \|\alpha\|^2 }{ \xi
\coth \pi \xi} \right) \cosh \pi \xi} .
 \end{equation}
In particular,
\begin{equation}
 I(\alpha,  \xi)=\frac{\exp\left(-\frac{2\pi \|\alpha\|^2 }{\xi \coth  
\pi \xi}\right)}{\cosh \pi \xi}.
\end{equation}
\end{theorem}
\begin{rem}
 While $I(\alpha,  \xi)$ characterizes the Gaussian measure, it does not
behave well under convolution.
\end{rem}

Along with the above characteristic function calculation the following moment calculation is used.
\begin{lemma}\label{H_star_moment_lemma}
 Let $\eta$ be a two dimensional Gaussian with identity covariance.  For each $k \geq 1$, and $N \geq 2$,
 \begin{equation}
  \E_{\eta^{\otimes N}}\left[H^{*}(\uw)^{2k}\right] \leq \mu_{2k}^2 \frac{N^{2k}}{2^{2k}}
 \end{equation}
where $\mu_{2k} = \frac{(2k)!}{2^k k!}$ is the $2k$th moment of a standard one dimensional Gaussian.

For any compactly supported probability measure $\mu$ of mean zero on $\bR^2$, for any $k \geq 1$, as $N \to \infty$,
\begin{equation}
 \E_{\mu^{\otimes N}}\left[H^{*}(\uw)^{2k}\right] \leq O_{k,\mu}\left(N^{2k}\right).
\end{equation}

\end{lemma}

\begin{proof}
 Write 
 \begin{equation}
  H^*(\uw) = \frac{1}{2}\sum_{1 \leq i \neq j \leq N} (-1)^{\one(i > j)} w_i^{(1)}w_j^{(2)}
 \end{equation}
and expand the moment to find
\begin{align}
 &\E_{\eta^{\otimes N}}\left[H^{*}(\uw)^{2k}\right] \\ \notag &\leq \frac{1}{2^{2k}}\E_{\eta^{\otimes N}}\left[\sum_{1 \leq m_1, ..., m_{2k}, n_1, ..., n_{2k} \leq N} w_{m_1}^{(1)}\cdots w_{m_{2k}}^{(1)}w_{n_1}^{(2)}\cdots w_{n_{2k}}^{(2)} \right]\\
&\notag= \frac{1}{2^{2k}}\left[\E_{\eta^{\otimes N}}\left[ \sum_{1 \leq m_1, ..., m_{2k} \leq N} w_{m_1}^{(1)}\cdots w_{m_{2k}}^{(1)}\right]\right]^2\\
&\notag = \mu_{2k}^2 \frac{N^{2k}}{2^{2k}}.
\end{align}

When treating general $\mu$ of compact support,
\begin{align}
 &\E_{\mu^{\otimes N}}\left[H^{*}(\uw)^{2k}\right] \\ \notag &= \frac{1}{2^{2k}}\E_{\mu^{\otimes N}}\left[\sum_{1 \leq m_1, ..., m_{2k}, n_1, ..., n_{2k} \leq N} \varepsilon_{\um, \un} w_{m_1}^{(1)}\cdots w_{m_{2k}}^{(1)}w_{n_1}^{(2)}\cdots w_{n_{2k}}^{(2)} \right].
\end{align}
with $\varepsilon_{\um, \un} \in \{-1, 0, 1\}$.
The expectation vanishes unless each index in $[N]$ which appears among $m_1, ..., m_{2k}, n_1, ..., n_{2k}$ appears at least twice.  There are $O(N^{2k})$ ways to choose which indices appear and $O_k(1)$ ways to assign $m_1,  ..., n_{2k}$ to the indices which appear.  For those assignments which don't vanish, the expectation is $O_{k, \mu}(1)$ by the compact support.

\end{proof}

We make the following convention regarding rare events. 
Say that a sequence of measurable events $\{A_{N}\}_{N \geq 1}$ such that
$A_N \subset S^N$ occurs \emph{with high probability} (w.h.p.) if the
complements satisfy the decay estimate,
\begin{equation}
\forall \; C\geq 0, \qquad \mu^{\otimes N}\left(A_N^c\right) =
O_C\left(N^{-C}\right)
\end{equation}
as $N \to \infty$.  The sequence of complements is said to be
\emph{negligible}. 
A sequence of events $\{A_N\}$ which is negligible
for $\mu^{\otimes N}$ is also negligible when $\mu^{\otimes N}$ is conditioned
on a non-negligible sequence of events $\{B_N\}$.

\subsection{The large spectrum}\label{large_spectrum_section}
Let $\mu$ be a mean 0, compactly supported probability measure on $\bR^2$.  
For $0<\vartheta<1$, define the large spectrum of $\hat{\mu}$ to be
\begin{equation}
\Spec_{\vartheta}(\mu) = \left\{\alpha \in \bR^2: |\hat{\mu}(\alpha)| > 1-\vartheta \right\}
\end{equation}
 and let
 \begin{equation}
\sM_\vartheta(\mu) = \{\alpha \in \Spec_\vartheta(\mu): \left|\hat{\mu}(\alpha)\right| \text{ is a local maximum}\}.
\end{equation}

Let
\begin{equation}
\check{\mu}(A) = \mu(\{x^{-1}: x \in A\})
\end{equation}
and set $\mu_2 = \mu * \check{\mu}$.  The measure $\mu_2$ is still mean 0, compactly supported and satisfies 
\begin{equation}\hat{\mu}_2(\alpha) = \int_{\bR^2}\cos(2\pi \alpha\cdot x)d\mu_2(x) = \left|\hat{\mu}(\alpha)\right|^2,\end{equation}
so $\Spec_\vartheta(\mu) = \Spec_{2\vartheta -\vartheta^2}(\mu_2)$ and $\sM_{\vartheta}(\mu) = \sM_{2\theta -\theta^2}(\mu_2)$.

For a differential operator $D_\beta = D_{i_1}D_{i_2}\cdots D_{i_\ell}$, set $|\beta| = \ell$.
\begin{lemma}\label{large_spectrum_lemma}
	Let $0\leq\vartheta \leq 1$, let $\alpha \in \Spec_{\vartheta}(\mu_2)$ and let $D_\beta$ be a differential operator.  Then
	\begin{equation}
	D_\beta \hat{\mu}_2(\alpha) = \left\{\begin{array}{lll}O_{\beta}\left(\vartheta^{\frac{1}{2}}\right) && |\beta| \text{ odd}\\ D_\beta \hat{\mu}_2(0) + O_\beta(\vartheta) && |\beta| \text{ even}\end{array}\right..
	\end{equation}
\end{lemma}

\begin{proof}
	Let $D_\beta = D_{i_1}\cdots D_{i_\ell}$. Differentiating under the integral, if $\ell$ is odd then
	\begin{equation}
	D_\beta \hat{\mu}_2(\alpha) = (-1)^{\one(\ell \equiv 1 \bmod 4)}(2\pi)^{\ell}\int_{\bR^2}x_{i_1} \cdots x_{i_\ell} \sin(2\pi \alpha \cdot x)d\mu_2(x)
	\end{equation}
	so that, using the compact support of $\mu_2$ and then Cauchy-Schwarz,
	\begin{align}
	\left|D_\beta \hat{\mu}_2(\alpha)\right|&\ll_\ell \int_{\bR^2}|\sin (2\pi \alpha \cdot x)|d\mu_2(x)\\
	\notag &= \int_{\bR^2}\sqrt{1-\cos^2(2\pi \alpha \cdot x)} d\mu_2(x)
	\\\notag& \leq \left( \int_{\bR^2} 1-\cos^2(2\pi \alpha \cdot x)d\mu_2(x)\right)^{\frac{1}{2}} \leq (2\vartheta)^{\frac{1}{2}}.
	\end{align}
	If $\ell$ is even, then
	\begin{align}
	&D_\beta \hat{\mu}_2(\alpha) = (-1)^{\frac{\ell}{2}} (2\pi)^\ell \int_{\bR^2}x_{i_1}\cdots x_{i_{\ell}}\cos(2\pi \alpha\cdot x)d\mu_2(x)\\
	\notag &= D_\beta \hat{\mu}_2(0) - (-1)^{\frac{\ell}{2}} (2\pi)^\ell \int_{\bR^2}x_{i_1}\cdots x_{i_\ell} (1-\cos(2\pi \alpha \cdot x))d\mu_2(x).
	\end{align}
	Again using the compact support, the integral in the last line is $O(\vartheta)$.
\end{proof}
	The previous lemma has the following consequences.
	\begin{lemma}\label{large_spectrum_structure_lemma}
		There is a constant $C_1 = C_1(\mu)$, $0 <C_1 < 1$ such that if $0< \vartheta<C_1$ then the following hold:
		\begin{enumerate}
			\item[i.] There are constants $C_2(\mu), C_3(\mu)>0$ such that if $\alpha_0 \in \sM_\vartheta(\mu_2)$ and $\|\alpha-\alpha_0\| < C_2$ then 
			\begin{equation}
			\hat{\mu}_2(\alpha) \leq \hat{\mu}_2(\alpha_0) - C_3\|\alpha-\alpha_0\|^2.
			\end{equation}
			\item[ii.]  There is a constant $C_4(\mu)>0$ such that if $\alpha \in \Spec_\vartheta(\mu_2)$ then there exists $\alpha_0 \in \sM_\vartheta(\mu_2)$ with
			\begin{equation}
			\|\alpha - \alpha_0\| \leq C_4 \sqrt{\vartheta}.
			\end{equation}
		\end{enumerate}
	Furthermore, if $\mu$ does not have a lattice component, then there is a growth function $\sF(\vartheta)$ tending to infinity as $\vartheta \downarrow 0$ such that, if $\alpha_0 \neq \alpha_1$ are distinct elements of $\sM_\vartheta(\mu_2)$ then \begin{equation}\|\alpha_0 - \alpha_1\| > \sF(\vartheta).\end{equation}
	\end{lemma}
\begin{proof}
	To prove i., Taylor expand about $\alpha_0$ using that the first derivatives vanish and that the third derivatives of $\hat{\mu}_2$ are uniformly bounded.  The term from the second degree Taylor expansion may be replaced with the corresponding term at $\alpha_0 = 0$, making an error which is $O(\vartheta)$.  This may be absorbed into the constant $C_3$ by making $C_1$ sufficiently small.
	
	To prove ii., first reduce $C_1(\mu)$ to guarantee that there is a ball $B_\delta(\alpha)$, $0 < \delta < 1$ a fixed constant, such that the maximum of $\hat{\mu}_2$ does not occur on the boundary of the ball.  This may be achieved by Taylor expanding about $\alpha$, which now includes a linear term, which is $O(\vartheta^{\frac{1}{2}})$.  Let $\alpha_0$ be the global maximum in the interior, and now apply part i. and $\hat{\mu}_2(\alpha_0)-\hat{\mu}_2(\alpha) \leq \vartheta$ to conclude that $\|\alpha-\alpha_0\| \ll \sqrt{\vartheta}$.
	
	To prove the final statement, note that for $0 \leq \vartheta \leq \frac{1}{4}$, if $\alpha_0, \alpha_1 \in \Spec_\vartheta(\mu_2)$ then $\alpha_0 - \alpha_1 \in \Spec_{4\vartheta}(\mu_2)$, see \cite{TV06}, p. 183.  An easier proof is possible here since the spectrum is positive, indeed,
	\begin{align}
	1-\hat{\mu}_2(\alpha_0-\alpha_1) =& \int_{\bR^2} 1- \cos (2\pi \alpha_0 \cdot x) \cos (2\pi \alpha_1 \cdot x) d\mu_2 \\\notag &- \int_{\bR^2}\sin (2\pi \alpha_0 \cdot x) \sin (2\pi \alpha_1 \cdot x) d\mu_2.
	\end{align}
	 Bound the first integral by
	 \begin{equation}
	 \int_{\bR^2}1 - \cos(2\pi \alpha_0 \cdot x)d\mu_2 + \int_{\bR^2}\cos(2\pi \alpha_0 \cdot x) (1-\cos(2\pi \alpha_1 \cdot x))d\mu_2 \leq 2 \vartheta.
	 \end{equation}
	 By Cauchy-Schwarz, the second integral is bounded in size by
	 \begin{equation}
	 \left(\int_{\bR^2} 1 - \cos^2(2\pi \alpha_0 \cdot x)d\mu_2 \int_{\bR^2} 1- \cos^2(2\pi \alpha_1 \cdot x)d\mu_2\right)^{\frac{1}{2}} \leq 2 \vartheta.
	 \end{equation}

	 The claim now follows on considering $\hat{\mu}_2(\alpha)$ in growing balls about 0.
\end{proof}

The following lemma gives information about variation of the phase of $\hat{\mu}(\alpha)$ in the large spectrum.
\begin{lemma}\label{phase_variation_lemma}
	Let $\mu$ be a measure of mean 0 and compact support on $\bR^2$, let $0 \leq \vartheta \leq \frac{1}{2}$ and let $\alpha_0 \in \sM_\vartheta(\mu)$. The following hold.
	\begin{enumerate}
		\item[i.] $\IM D_i \log \hat{\mu}(\alpha_0) = O_\mu(\vartheta)$
		\item[ii.] $\IM D_i D_j \log \hat{\mu}(\alpha_0) = O_\mu(\sqrt{\vartheta})$.
		\item[iii.] For all $\alpha \in \Spec_{\frac{1}{2}}(\mu)$, 
		\begin{equation}		
		\IM D_{i_1}D_{i_2}D_{i_3} \log \hat{\mu}(\alpha) = O(1).
		\end{equation}
	\end{enumerate}
\end{lemma} 
\begin{proof}
	Let $\hat{\mu}(\alpha_0) = e_{\alpha_0}(\phi_0) |\hat{\mu}(\alpha_0)|$.
	
	For i.  
	\begin{equation}
	D_i \log \hat{\mu}(\alpha_0)=\frac{D_i \hat{\mu}(\alpha_0)}{\hat{\mu}(\alpha_0)} = \frac{2\pi i}{|\hat{\mu}(\alpha_0)|}\int_{\bR^2}x_i e_{\alpha_0}(x-\phi_0)d\mu(x).
	\end{equation}
	Since $\mu$ is mean 0,
	\begin{equation}
	\IM D_i \log \hat{\mu}(\alpha_0) = \frac{2 \pi}{|\hat{\mu}(\alpha_0)|}\int_{\bR^2} x_i (\cos(2\pi \alpha_0 \cdot (x-\phi_0))-1)d\mu(x).
	\end{equation}
	By the compact support,
	\begin{align}
	|\IM D_i \log \hat{\mu}(\alpha_0)|& \ll \int_{\bR^2}1 - \cos(2\pi \alpha_0 \cdot (x-\phi_0)) d\mu(x) \\ 
	&\notag = 1-|\hat{\mu}(\alpha_0)| \leq \vartheta.
	\end{align}
	
	For ii., write
	\begin{equation}
	D_iD_j \log \hat{\mu}(\alpha_0) = \frac{D_iD_j \hat{\mu}(\alpha_0)}{\hat{\mu}(\alpha_0)} - \frac{D_i \hat{\mu}(\alpha_0) D_j \hat{\mu}(\alpha_0)}{\hat{\mu}(\alpha_0)^2}.
	\end{equation}
	The subtracted term is real since $\frac{D_i \hat{\mu}(\alpha_0)}{\hat{\mu}(\alpha_0)}$ is imaginary ($\alpha_0$ is a maximum for $|\hat{\mu}(\alpha_0)|$).  Hence, again using the compact support and Cauchy-Schwarz,
	\begin{align}
	\IM D_i D_j \log \hat{\mu}(\alpha_0) &= \frac{-4\pi^2}{|\hat{\mu}(\alpha_0)|} \int_{\bR^2} x_i x_j \sin(2\pi \alpha_0 \cdot (x-\phi_0))d\mu(x)\\&
	\notag \ll \int_{\bR^2} \sqrt{1 - \cos^2(2\pi \alpha_0 \cdot (x-\phi_0))}d\mu(x)\\ \notag
	&\ll \left(\int_{\bR^2} 1 - \cos(2\pi \alpha_0 \cdot (x-\phi_0))d\mu(x)\right)^{\frac{1}{2}} \leq \vartheta^{\frac{1}{2}}.
	\end{align}
	
	To obtain iii., note that the first three derivatives of $\hat{\mu}$ are bounded due to the compact support.
\end{proof}

The results of this section are collected into the following lemma which permits approximating $\hat{\mu}(\alpha)$ in neighborhoods of a local maximum for $|\hat{\mu}(\alpha)|$.
\begin{lemma}\label{Taylor_expansion_near_local_max}
 Let $\mu$ be a probability measure on $\bR^n$ of covariance matrix $\sigma^2$.  There is a constant $C = C(\mu) > 0$ such that for all $0 \leq \vartheta \leq C$ and for all $\alpha_0 \in \sM_\vartheta(\mu)$ we have
	\begin{equation}\label{large_spectrum_Taylor_expansion}
	\hat{\mu}(\alpha_0 + \alpha) = \hat{\mu}(\alpha_0)\E[e_\alpha(X)] + O\left(\vartheta \|\alpha\|  + \|\alpha\|^3\right)
	\end{equation}
	with $X$ distributed as $\eta(0,\sigma)$.
\end{lemma}

\begin{proof}
 	Taylor expand $\log \frac{\hat{\mu}(\alpha_0+\alpha)}{\hat{\mu}(\alpha_0)}$ in a ball of constant radius about $\alpha = 0$ to find
	\begin{equation}
	\log \frac{\hat{\mu}(\alpha_0+\alpha)}{\hat{\mu}(\alpha_0)} = \frac{1}{2}\alpha^t H_0 \alpha  + O\left(\vartheta \|\alpha\| + \vartheta^{\frac{1}{2}}\|\alpha\|^2+ \|\alpha\|^3\right),
	\end{equation}
	with $H_0$  the Hessian of $\log \hat{\mu}(\alpha)$ at 0. In making this expansion, we've used the estimates for derivatives of $\hat{\mu}_2(\alpha_0  +\alpha)$ in Lemma \ref{large_spectrum_lemma}  together with
	\begin{equation}
	 \RE \log \frac{\hat{\mu}(\alpha_0+\alpha)}{\hat{\mu}(\alpha_0)} = \frac{1}{2} \log \frac{\hat{\mu}_2(\alpha_0+\alpha)}{\hat{\mu}_2(\alpha_0)}
	\end{equation}
	and the estimates for derivatives of $\IM \log \hat{\mu}(\alpha_0 + \alpha)$ in Lemma \ref{phase_variation_lemma}.
	
	 Then (we've absorbed the $\vartheta^{\frac{1}{2}}\|\alpha\|^2$ error term into the others)
	\begin{equation}
	\hat{\mu}(\alpha_0 + \alpha) = \hat{\mu}(\alpha_0)\E[e_\alpha(X)] + O\left(\vartheta \|\alpha\|  + \|\alpha\|^3\right).
	\end{equation}
	Since $\hat{\mu}$ is bounded, this formula holds for all $\alpha$ by adjusting the constants appropriately.
	
\end{proof}

\section{Proof of Theorem
\ref{zero_return_theorem}}\label{summary_of_argument_section}
We first treat Theorem \ref{zero_return_theorem} which is illustrative of the main idea, before proving Theorem \ref{local_limit_theorem}.
Identify $\un = (n_1, n_2, n_3)^t \in \zed^3$, with $g_{\un} = [n_1, n_2, n_3] \in \bH(\zed)$ and let $\nu$ be the limiting Gaussian measure under convolution by $\mu$.

\begin{proposition} For each $\un = (n_1, n_2, n_3)^t \in \zed^3$,
\begin{align}\label{lattice_target}
&P_N(n_1, n_2, n_3):= \mu^{*N} \left(\{g_{\un}\}\right)= \frac{1}{N^2} \cdot
\frac{d\nu}{d g}\left(d_{\frac{1}{\sqrt{N}}} g_{\un}\right)  +
O\left(N^{-\frac{5}{2}} \right).
\end{align}
\end{proposition}
 
Recalling the multiplication rule
\begin{equation}
\prod_{i=1}^N \left[w_i^{(1)}, w_i^{(2)},0\right] = \left[\ouw^{(1)},
\ouw^{(2)},  \frac{1}{2}\ouw^{(1)}\ouw^{(2)}   + H^*(\uw) \right],
\end{equation}
which is valid for $w_i = [\pm 1, 0, 0]$ or $[0, \pm 1, 0]$, it suffices to calculate, with $\bU_N$ standing for the product measure $\mu^{\otimes N}$ and expectation with respect to $\bU_N$,
\begin{align}
&\bU_N\left(\ouw_{\ab} = (n_1, n_2)^t, H^{*}(\uw) = n_3 - \frac{1}{2}n_1n_2\right)=\\ \notag &  \int_{(\bR/\zed)^3}e_\alpha\left((n_1,n_2)^t \right)e_{\xi}\left(n_3 - \frac{n_1n_2}{2}\right)\E\left[\overline{e_\alpha(\ouw_{\ab})e_{\xi}\left( H^*(\uw)\right) }\right]d\xi d\alpha.
\end{align}

\subsection{Reduction to a central limit theorem}
The following two lemmas reduce to a quantitative central limit theorem by
truncating frequency space to the scale of the distribution.

\begin{lemma}\label{large_frequency_xi_lemma}
 For any $A > 0$ there is $C = C(A)>0$ such that if $\|\xi\|_{\bR/\zed}
\geq \frac{C \log N}{N}$, for all $\alpha \in \bR^2$,
\begin{equation}
 \left|\E_{\bU_N}\left[e_\alpha\left(\ouw_{\ab} \right)e_\xi\left( H^*(\uw)\right)
\right]\right| \leq N^{-A}.
\end{equation}
\end{lemma}
\begin{proof}
Choose $k = k(\xi)$ according to the rule
\begin{equation}
 k(\xi) =\left \{\begin{array}{lll} 1, && |\xi| > \frac{1}{10}, \\  \left\lfloor
\frac{1}{2|\xi|} \right \rfloor,&& |\xi| \leq \frac{1}{10}\end{array}\right..
\end{equation}
Let $N' = \left \lfloor \frac{N}{2k} \right \rfloor$. The group $G_k = C_2^{N'}$ acts on strings
of length $N$ with $j$th factor exchanging the order of the substrings
of length $k$ ending at $(2j-1)k$ and $2jk$. 

Given string $\uw$, write
$\huw$ for the string of length $2N'$ with $j$th entry given by
\begin{equation}
 \hw_j = \sum_{i=1}^k w_{(j-1)k+i}.
\end{equation}
Write
\begin{equation}
 H^*(\uw) = H_k^1(\uw) + H_k^2(\uw), \qquad H_k^2(\uw) =
\sum_{j=1}^{N'}H^*\left(\hw_{2j-1}, \hw_{2j}\right).
\end{equation}
Both $\ouw_{\ab}$ and $H_k^1$ are invariant under $G_k$.  Exchanging the order of the expectations, which is justified because the group action is finite,
\begin{align}
& \E_{\bU_N}\left[e_{\alpha}\left(\ouw_{\ab} \right)e_\xi \left(H^*(\uw)\right) 
\right] = \E_{\utau \in G_k} \left[\E_{ \bU_N}\left[e_{\alpha}\left(\ouw_{\ab} \right)e_\xi \left(H^*(\utau \cdot \uw)\right) 
\right] \right]
\\\notag&=
\E_{\bU_N}\left[e_{\alpha}(\ouw_{\ab})e_\xi \left(H_k^1(\uw)\right)\E_{\utau \in G_k}\left[e_\xi \left(H_k^2(\utau \cdot
\uw)\right)\right]
\right],
\end{align}
and, therefore,
\begin{align}
\left|\E_{\bU_N}\left[e_{\alpha}(\ouw_{\ab})e_\xi \left(H^*(\uw)\right)
\right]\right| &\leq \E_{\bU_N}\left[\left|\E_{\utau \in
G_k}\left[e_\xi\left( H_k^2(\utau \cdot\uw)\right)\right]\right|
\right].
\end{align}
By Cauchy-Schwarz,
\begin{equation}
\E_{\bU_N}\left[\left|\E_{\utau \in G_k}\left[e_\xi\left(
H_k^2(\utau \cdot\uw)\right)\right]\right|
\right]^2 \leq \E_{\bU_N}\left[\left|\E_{\utau \in G_k}\left[e_\xi\left(
H_k^2(\utau \cdot\uw)\right)\right]\right|^2 \right].
\end{equation}
One checks, using the product group structure,
\begin{equation}
 \left|\E_{\utau \in G_k}\left[e_\xi\left(
 H_k^2(\utau \cdot\uw)\right)\right]\right|^2 = \prod_{j=1}^{N'}\left(\frac{1+\cos\left(2\pi \xi
H^*(\hw_{2j-1}  , \hw_{2j}) \right)}{2}\right),
\end{equation}
and hence, since the coordinates in $\uw$ are i.i.d.,
\begin{align}
&\E_{\bU_N}\left[\left|\E_{\utau \in G_k}\left[e_\xi\left(
H_k^2(\utau \cdot\uw)\right)\right]\right|^2 \right] \\& \notag = \left(\frac{1 + \E_{\bU_N}\left[\cos\left(2\pi \xi
	H^*(\hw_{1}  , \hw_{2}) \right)\right]}{2}\right)^{N'}.
\end{align}
By Lemma \ref{H_star_decay} the expectation in the cosine is uniformly bounded in size by $1-c(\mu)$ for some $c(\mu)>0$.  The claim is completed by using the estimate $(1-x)^{N'} \leq e^{-N' x}$, which is valid for $0 \leq x \leq 1$.
\end{proof}

The following lemma obtains cancellation in $\alpha$.
\begin{lemma}\label{large_freq_abel_lemma}
Let $A, \epsilon > 0$ and $0 \leq \|\xi\|_{\bR/\zed} \leq \frac{C \log N}{N}$
where $C$ is as in Lemma \ref{large_frequency_xi_lemma}.  For all $N$
sufficiently large, if
$
\|\alpha\|_{\bR^2/\zed^2} \geq N^{\epsilon -
\frac{1}{2}},
$
then
\begin{equation}
\left| \E_{\bU_N} \left[e_{\alpha}\left(  \ouw_{\ab}\right)e_\xi\left( H^*(\uw)\right) \right]\right| \leq N^{-A}.
\end{equation}
\end{lemma}

\begin{proof}
 Let $N' = \left \lfloor
N^{1-\epsilon}\right \rfloor$.  Let $\uw_0$ be $\uw$ truncated at $N'$ and
let $\uw_t$ be the remainder of $\uw$ so that $\uw$ is the concatenation $\uw_0
\oplus \uw_t$. 
 Write
\begin{equation}H^*(\uw) = H^*(\uw_0) + H^*(\ouw_0, \ouw_t) + H^*(\uw_t)\end{equation} to bound
\begin{align}
&\left| \E_{\bU_N} \left[e_{\alpha}\left(  \ouw_{\ab}\right)e_\xi\left( H^*(\uw)\right) \right]\right|
\\&\notag \leq \E_{\uw_t \sim \mu^{\otimes (N-N')}}\left[\left|\E_{\uw_0 \sim \mu^{\otimes N'}} \left[e_\alpha(\ouw_{0, \ab}) e_\xi(H^*(\uw_0) + H^*(\ouw_0, \ouw_t)) \right]  \right| \right].
\end{align}
Truncate the outer integral to $\|\ouw_t\| \leq \sqrt{N}
\log N$, which holds w.h.p. 
Let $E_k(x)$ denote the degree $k$ Taylor
expansion of $e_1(x)$ about 0, and recall that the integral form of Taylor's theorem gives
\begin{align}
 \left|e_1(x) - E_k(x)\right| &= \left| (2\pi x)^{k+1} \int_0^1 \frac{(1-t)^k}{k!} e_1(xt) dt \right| \leq \frac{(2\pi |x|)^{k+1}}{(k+1)!}.
\end{align}
  Use Lemma \ref{H_star_moment_lemma} to choose $k = k(A, \epsilon)$ be odd and sufficiently large 
so that
\begin{align}
 &\E_{\uw_0 \sim \mu^{\otimes N'}}\left[\left|E_k\left(\xi H^*(\uw_0) \right) - e_\xi( H^*(\uw_0)) \right| \right] \\
 \notag & \leq \frac{(2\pi |\xi|)^{k+1}}{(k+1)!}\E_{\uw_0 \sim \mu^{\otimes N'}}\left[|H^*(\uw_0)|^{k+1}  \right] \\ \notag &\leq O_{k, \mu}(|\xi| N)^{k+1} \leq \frac{1}{2N^A}.
\end{align}

It thus 
suffices to 
estimate 
\begin{align}
 \E_{\bU_{N'}}\Biggl[
e_{\alpha}\left(\ouw_{0,\ab}\right)e_\xi \left(H^*(\ouw_0,
\ouw_t)\right) E_k\left(\xi H^*(\uw_0)\right)& \Biggr].
\end{align}
Expand $E_k$ into $\Poly(N)$ terms, each depending on boundedly many indices
from $\uw_0$.  Expectation over the remaining terms factors as a product which
is exponentially small in a power of $N$, hence negligible. 
\end{proof}

\subsection{Quantitative Gaussian approximation}
In the range $\|\alpha\| \leq N^{\epsilon - \frac{1}{2}}$, $|\xi| \ll \frac{\log
N}{N}$, expectation with respect to $\mu$ is replaced with expectation taken
over a measure with projection to the abelianization given by a Gaussian of the
same covariance matrix as $\mu_{\ab}$. The modified characteristic
function in the Gaussian case is evaluated precisely in Theorem
\ref{char_fun_theorem}, which finishes the proof.

Let $\sigma^2$ be the covariance matrix of $\mu_{\ab}$ and let $\eta(0,\sigma)$ be a centered Gaussian of the same covariance. Set $\delta = \det \sigma$. Taylor expand $\log \hat{\mu}(\beta)$ about $\beta = 0$ to find a cubic map $T(\beta)$ such that
\begin{equation}
 \hat{\mu}_{\ab}(\beta) = \hat{\eta}(\beta)\left(1 + T(\beta)\right) + O\left(\|\beta\|^4\right).
\end{equation}
In the phase $e_{\alpha}\left(\ouw_{\ab}\right)e_\xi\left(H^*(\uw)\right)$ let
\begin{equation}
 \alpha_j(\uw) = \alpha + \frac{\xi}{2}\left[\sum_{i \neq j}(-1)^{\one(i<j)}w_i^{(2)}, \sum_{i \neq j}(-1)^{\one(i>j)}w_i^{(1)} \right]^t
\end{equation}
so that $\alpha_j(\uw)\cdot w_j$ is the part which depends on $w_j$.
The Gaussian replacement scheme is as follows.
\begin{lemma}\label{Gaussian_replacement_lemma}
 Let $0 < \epsilon < \frac{1}{2}$ and $C>0$ be constants.  For $\|\alpha\| \leq N^{\epsilon - \frac{1}{2}}$ and $|\xi| \leq \frac{C\log N}{N}$,
 \begin{align}
  &\E_{\bU_N} \left[e_{\alpha}\left(\ouw_{\ab}\right) e_\xi \left(H^*(\uw)  \right)  \right]= I\left(\sqrt{N}\sigma \cdot \alpha, N\delta \xi \right)+ O\left(N^{-1+O(\epsilon)}\right)\\ \notag &+ \E_{\eta^{\otimes N}}\left[e_{\alpha}\left(\ouw_{\ab}\right) e_\xi \left(H^*(\uw)  \right)\left( \sum_j T(\alpha_j(\uw)) \right) \right].
 \end{align}

\end{lemma}

\begin{proof}
Since $\E_{\eta^{\otimes N}}\left[e_{\alpha}(\ouw_{\ab})e_\xi(H^*(\uw)) \right]= I\left( {N}^{\frac{1}{2}}\sigma \alpha,  N \delta \xi;N\right)$, and since in the stated range of $\alpha, \xi$,
\begin{equation}I\left( {N}^{\frac{1}{2}}\sigma \alpha,  N \delta \xi;N\right) = I\left( {N}^{\frac{1}{2}}\sigma \alpha,  N \delta \xi\right) + O\left(N^{-1 + O(\epsilon)}\right)\end{equation} by Theorem \ref{char_fun_theorem}, it suffices to prove
\begin{align}\label{degree_three_approx}
 &\E_{\bU_N} \left[e_{\alpha}\left(\ouw_{\ab}\right) e_\xi \left(H^*(\uw)  \right)  \right]+ O\left(N^{-1+O(\epsilon)}\right)\\
 \notag &=\E_{\eta^{\otimes N}}\left[e_{\alpha}\left(\ouw_{\ab}\right) e_\xi \left(H^*(\uw)  \right)\left(1+ \sum_j T(\alpha_j(\uw)) \right) \right].
\end{align}

For convenience, write
\begin{equation}
 T_j(\alpha, \xi, \uw) = T(\alpha_j(\uw))
\end{equation}
and, for $k \neq j$, 
\begin{equation}
 T_j(\alpha, \xi, \uw) = T_j^k(\alpha, \xi, \uw) + \hat{T}_j^k(\alpha, \xi, \uw)
\end{equation}
in which $T_j^k$ collects monomials in $T_j$ which depend on $w_k$, and $\hat{T}_j^k$ collects monomials which don't depend on $w_k$.

Since the expectation does not depend upon the third
coordinate, write $\mu_{\ab}^{\otimes N}$  in place of
$\bU_N$. 
For $0 \leq j \leq N$ consider the measure $\mu_j = \mu_{\ab}^{\otimes
j}\otimes \eta^{\otimes (N-j)}$ in which the first $j$ coordinates are i.i.d.
with measure $\mu_{\ab}$ and last $N-j$ coordinates are independent of the first $j$ and are i.i.d. $\eta$.

We prove (\ref{degree_three_approx}) iteratively by showing that, for each $k \geq 1$,
\begin{align}\label{iterative_bound}
&S_k:= \E_{\mu_k}\left[e_\alpha(\ouw_{\ab}) e_\xi(H^*(\uw))  \left(1 + \sum_{j>k} T_j(\alpha, \xi, \uw)\right) \right]\\&\notag = O\left(N^{-2 + O(\epsilon)}\right) + \E_{\mu_{k-1}}\left[e_\alpha(\ouw_{\ab}) e_\xi(H^*(\uw))  \left(1 + \sum_{j>k-1} T_j(\alpha, \xi, \uw)\right) \right]
\\&\notag = O\left(N^{-2 + O(\epsilon)}\right) + S_{k-1},
\end{align}
which suffices since (\ref{degree_three_approx}) may be written as $|S_N- S_0| = O\left(N^{-1 + O(\epsilon)}\right)$.  By the triangle inequality, and setting apart expectation in the $k$th variable as the inner integral,
\begin{align}
 &|S_k - S_{k-1}|\\ &\notag\leq \E_{\mu_{\ab}^{\otimes (k-1)}\otimes \eta^{\otimes (N-k)}}\left|\int_{w_k}e_{\alpha_k(\uw)}(w_k) (d \mu_{\ab} - \left(1 +  T_k(\alpha, \xi, \uw)\right) d\eta) \right|\\\notag&+\E_{\mu_{\ab}^{\otimes (k-1)}\otimes \eta^{\otimes (N-k)}}\left|\int_{w_k}e_{\alpha_k(\uw)}(w_k)\left(\sum_{j>k}T_j(\alpha, \xi, \uw) \right) (d \mu_{\ab} - d\eta)   \right|.
\end{align}
In the first line of the right hand side, note that $T_k(\alpha, \xi, \uw)$ does not depend on $w_k$, so that Taylor expanding the exponential obtains a bound of  $O(\|\alpha_k(\uw)\|^4)$, which suffices since \begin{equation}\E_{\mu_{\ab}^{\otimes (k-1)}\otimes \eta^{\otimes (N-k)}} \left[\|\alpha_k(\uw)\|^4\right] = O\left(N^{-2 + 4\epsilon}\right).\end{equation}
In the second line, write $T_j = T_j^k + \hat{T}_j^k$.  Since $\hat{T}_j^k$ does not depend on $w_k$, matching the first two moments of $\mu_{\ab}$ and $\eta$ gives
\begin{align}
& \E_{\mu_{\ab}^{\otimes (k-1)}\otimes \eta^{\otimes (N-k)}}\left|\int_{w_k}e(\alpha_k(\uw)\cdot w_k) \left(\sum_{j>k}\hat{T}_j^k(\alpha, \xi, \uw) \right) (d \mu_{\ab} - d\eta)   \right|
\\& \notag \ll \E_{\mu_{\ab}^{\otimes (k-1)}\otimes \eta^{\otimes (N-k)}} \left[\|\alpha_k(\uw)\|^3 \left|\sum_{j > k}\hat{T}_j^k(\alpha, \xi, \uw)\right| \right] \ll N^{-2 + 6\epsilon}.
\end{align}
Finally, to bound the terms from $T_j^k$, Taylor expand $e(\alpha_k(\uw)\cdot w_k)$ to degree 2 to bound
\begin{align}
& \E_{\mu_{\ab}^{\otimes (k-1)}\otimes \eta^{\otimes (N-k)}}\left|\int_{w_k}e(\alpha_k(\uw)\cdot w_k) \left(\sum_{j>k}T_j^k(\alpha, \xi, \uw) \right) (d \mu_{\ab} - d\eta)   \right|
\\& \notag \ll \E_{\mu_{\ab}^{\otimes (k-1)}\otimes \eta^{\otimes (N-k)}} \left|\int_{w_k}\left(1 + 2\pi i \alpha_k(\uw)\cdot w_k\right)\sum_{j > k}T_j^k(\alpha, \xi, \uw)(d\mu_{ab}-d\eta)\right| 
\\& \notag + \E_{\mu_{\ab}^{\otimes (k-1)}\otimes \eta^{\otimes (N-k)}} \int_{w_k}\|\alpha_k(\uw)\|^2|w_k|^2 \left|\sum_{j > k}T_j^k(\alpha, \xi, \uw)\right| (d\mu_{\ab} + d\eta).
\end{align}
Since the first two moments of $\mu_{\ab}$ and $\eta$ match, the only terms which survive the first line here are degree 3 in $w_k$, and these contribute $ O(N^{-2 + 3\epsilon})$.  In the second line here, keeping in mind that $T_j^k$ contains only monomials that have a factor of $\xi w_k^{(1)}$ or $\xi w_k^{(2)}$, one obtains a bound of $O(N^{-2 +5\epsilon} )$ by applying Cauchy-Schwarz to separate the integrands. This completes the iterative estimate (\ref{iterative_bound}). 
\end{proof}

We give two estimates for the error term
\begin{equation}\label{cubic_terms}
 \sT = \E_{\eta^{\otimes N}}\left[e_{\alpha}(\ouw_{\ab})e_\xi(H^*(\uw))\left(\sum_j T(\alpha_j(\uw)) \right) \right]
\end{equation}
depending on the relative sizes of $\alpha$ and $\xi$. In this part of the argument we assume that $\eta(0, \sigma)$ has identity covariance, which may be achieved by rescaling $\alpha$ and $\xi$ by constant factors.

\begin{lemma}\label{large_xi_error_lemma}
  There exists $c > 0$ such that, for $\|\alpha\| \leq
N^{\epsilon - \frac{1}{2}}$ and $|\xi| \ll \frac{\log N}{N}$,
 \begin{align}
 \sT = O \left(N\|\alpha\|^3 + N^{\frac{5}{2}}|\xi|^3 \right) \exp\left(-cN|\xi| \right).%e_{\xi}\left(-N \overline{\tilde{z}} \right)
 \end{align}

\end{lemma}
\begin{proof}
  Bound each term
\begin{equation}\label{individual_terms}
\E_{\eta^{\otimes N}}\left[e_{\alpha}(\ouw_{\ab})e_\xi(H^*(\uw)) T(\alpha_j(\uw)) \right]
\end{equation}
individually by setting 
\begin{equation}
k = \left\{\begin{array}{ccc} \left\lfloor \frac{1}{2|\xi|}\right \rfloor, && |\xi|N > 1\\\\
            \left\lfloor \frac{N}{2}\right\rfloor, && \text{otherwise}
           \end{array}\right. 
\end{equation}
 and allowing $G_k$ to act as in Lemma \ref{large_frequency_xi_lemma}.  Let $G_k'$ be the subgroup omitting the factor that moves $w_j$.  Then $G_k'$ leaves $T(\alpha_j(\uw))$ invariant, so that
\begin{equation*}
(\ref{individual_terms})=\E_{\eta^{\otimes N}}\left[e_{\alpha}(\ouw_{\ab})\E_{\utau \in G_k'}[e_\xi(H^*(\utau \cdot \uw))] T(\alpha_j(\uw)) \right].
\end{equation*}
By Cauchy-Schwarz, 
\begin{align*}
|(\ref{individual_terms})|^2 & \leq \E_{\eta^{\otimes N}}\left[|\E_{\utau \in G_k'}[e_\xi(H^*(\utau \cdot \uw))] |^2\right]\E_{\eta^{\otimes N}}[|T(\alpha_j(\uw))|^2].
\end{align*}
Arguing as in Lemma \ref{large_frequency_xi_lemma} now obtains the estimate, for some $c>0$,
\begin{equation}
|(\ref{cubic_terms})| \ll \left(N\|\alpha\|^3 + N^{\frac{5}{2}}|\xi|^3\right) \exp\left(-c |\xi|N\right).
\end{equation}
\end{proof}

To obtain  decay in $\|\alpha\|$ instead of $|\xi|$, consider the  degree 3 polynomial
$
  \sum_j T(\alpha_j(\uw))$ which consists of monomials of which
\begin{enumerate}
 \item[i.]  Those constant in $\uw$ and  cubic
in $\alpha$ have absolute sum of coefficients $O(N)$.
 \item[ii.] Those linear in $\xi\uw$ and quadratic in $\alpha$ have absolute sum of coefficients  $O(N^2)$.
 \item[iii.] Those quadratic in $\xi \uw$ and linear in $\alpha$ have absolute sum of coefficients $O(N^3)$.  Of
these, those with a  repeated factor from $\uw$ have absolute sum of coefficients  $O(N^2)$.
 \item[iv.] Those that are cubic in $\xi \uw$ have absolute sum of coefficients $O(N^4)$. Of these, those with a repeated factor from $\uw$ have absolute sum of coefficients $O(N^3)$.
\end{enumerate}
Write $M$ for the typical monic monomial, so that $M$ is of form \begin{equation}1, w_{i_1}^{(\epsilon_1)}, w_{i_1}^{(\epsilon_1)}w_{i_2}^{(\epsilon_2)}, w_{i_1}^{(\epsilon_1)}w_{i_2}^{(\epsilon_2)}w_{i_3}^{(\epsilon_3)}\end{equation} with $\epsilon_j \in \{1,2\}$, according as the case is i., ii., iii. or iv..
Given a typical monomial $M$ of $T$, 
write $\omega(M)$ for the number of
variables
from $\uw$ which are odd degree in $M$.

\begin{lemma}\label{large_alpha_error_lemma}
There is a constant $c>0$ such that, for $0<|\xi| \ll \frac{\log N}{N}$ and $\sqrt{|\xi|} \leq \|\alpha\| \leq N^{\epsilon - \frac{1}{2}}$,
 \begin{align}\label{alpha_decay}
 &\sT=\biggl[
O\left(\|\alpha\|(1 + N\|\alpha\|^2)(1 + N^3|\xi|^3)\right)\\\notag & \qquad
\qquad\times \exp\left(-c \left(\|\alpha\|^2 \min\left(N,
\frac{1}{N|\xi|^{2}}\right)\right)\right)\biggr].
 \end{align}

\end{lemma}

\begin{proof}

Consider the expectation
\begin{equation}
E_M = \E_{\eta^{\otimes N}}\left[M e_{\alpha}\left(\ouw_{\ab}\right)e_\xi\left(
H^*(\uw) \right)
\right].
\end{equation}
We show that, for some $c > 0$,
\begin{equation}\label{monomial_bound}
 E_M \ll \|\alpha\|^{\omega(M)} \exp\left(-c \left(\|\alpha\|^2\min\left(N, \frac{1}{N|\xi|^2}\right)\right)\right),
\end{equation}
which suffices on summing over the monomials described in i. through iv. above.

Let $c_1 > 0$
be a small constant, and let 
\begin{equation}\label{def_N_prime}
N' = \left\{\begin{array}{lll}\left\lfloor\frac{c_1}{|\xi|}\right\rfloor,&& N|\xi| >c_1\\ N, && \text{otherwise}\end{array}\right..  
\end{equation}
Let
$\uw_0$ be the initial string of $\uw$ of length $N'$ and assume that this
includes any variables from $M$; the general case may be handled by a
straightforward modification. Write $\uw = \uw_0 \oplus \uw_t$ so that $\uw_t$ contains
the remaining variables.  Write \begin{equation}H^*(\uw) =
H^*(\uw_0) + H^*(\overline{\uw}_0, \overline{\uw}_t) + H^*(\uw_t).\end{equation} Write 
$\ha =
 \alpha + \frac{\xi}{2}\left[\ouw_t^{(2)}, -\ouw_t^{(1)}\right]^t$. Bound
\begin{equation}
 |E_M| \leq \E_{\uw_t \sim \eta^{\otimes(N-N')}}\left[\left|\E_{\uw_0 \sim \eta^{\otimes N'}}\left[M e_{\ha}\left(\ouw_0 \right)e_\xi\left(H^*(\uw_0)\right) \right]   \right|\right].
\end{equation}

 Expand 
$e_\xi( 
H^*(\uw_0))$ in
Taylor series to degree $L:= \left\lfloor N^{2\epsilon}\right \rfloor$.  The
error in doing so is bounded by 
\begin{align}
 \frac{(2\pi|\xi|)^{L+1}}{(L+1)!} \E_{\uw_0 \sim \eta^{\otimes N'}}\left[|M| |H^*(\uw_0)|^{L+1}\right].
\end{align}
Apply Cauchy-Schwarz to remove the monomial $M$,  then insert the moment bound of Lemma \ref{H_star_moment_lemma} to estimate this by
\begin{align}
& \ll \frac{(2\pi|\xi|)^{L+1}}{(L+1)!} \E_{\uw_0 \sim \eta^{\otimes N'}}\left[H^*(\uw_0)^{2L+2}\right]^{\frac{1}{2}}\\
& \notag \leq \frac{(2\pi|\xi|)^{L+1}}{(L+1)!} \frac{(2L+2)!}{(L+1)! 2^{2L+2}} (N')^{L+1}\\ \notag
& \leq (2\pi |\xi|N')^{L+1} \leq (2\pi c_1)^{L+1}.
\end{align}
If $c_1 < \frac{1}{2\pi}$ then this is bounded by, for some $c>0$, $\exp\left(-c N^{2\epsilon}\right).$

In the Taylor expansion, expectation over $\uw_0$ is bounded by
\begin{align}&
 \sum_{\ell=0}^L \frac{(2\pi
|\xi|)^\ell}{\ell!}\left|\E_{\eta^{\otimes N'}}\left[Me_{\ha}(\ouw_0)
H^*(\uw_0)^\ell \right] \right|\\
\notag& \leq \sum_{\ell=0}^L \frac{(2\pi |\xi|)^\ell}{2^\ell \ell!}\sum_{\um, \un \in
[N']^\ell}\left| \E_{\eta^{\otimes N'}}\left[Me_{\ha}( \ouw_0)
 w_{m_1}^{(1)}\cdots w_{m_\ell}^{(1)}
w_{n_1}^{(2)}\cdots w_{n_\ell}^{(2)}\right]\right|.
\end{align}

The expectation factors as a product.  Those indices of $[N']$ which do not have a monomial factor contribute, for some $c_2 > 0$,
\begin{equation}
\leq \exp\left(-2\pi^2 (N'- 2\ell-3) \|\ha\|^2 \right) \leq \exp\left(-c_2 N'\|\ha\|^2 \right). 
\end{equation}

Let $\sE$ (resp. $\sO$) be those indices in $[N']$ which appear a positive even (resp. odd) number of times among the factors in $M$ and $m_1, \cdots, m_\ell, n_1, \cdots, n_\ell$. 

For indices which appear a positive even $h_j$ number of times, bound $|e_{\ha}(w_j)| \leq 1$, so that the expectation is bounded by the $h_j$th moment of a 1-dimensional Gaussian, 
\begin{equation}
 \mu_{h_j} = 
\frac{h_j!}{2^{\frac{h_j}{2}} \left( \frac{h_j}{2}\right)!}.
\end{equation}

At indices which appear an odd $h_j$ number of times, set $e_{\ha}(w_j) = 1 + O(\|\ha\| \|w_j\|)$.  Expectation against 1 vanishes.  The remaining expectation is bounded by
\begin{equation}
 \ll \|\ha\| \mu_{h_j+1}.
\end{equation}

The configurations in which no index outside $M$ appears with multiplicity greater than 2, and no more than one of the $m_j, n_j$ fall on an odd degree index of $M$ and none fall on an even degree index of $M$, make a dominant contribution. Call these \emph{base configurations}. The \emph{type} of a base configuration is described by a triple $(\up, \ell_1, \ell_2)$ where $\up$ indicates whether an index from $\um, \un$ falls on each odd degree index present in $M$, where $\ell_1$ counts the number of indices which appear once, and $\ell_2$ counts the number of indices which appear twice.  Let $|\up|$ be the number of indices which fall on $M$.  Thus
\begin{equation}
 2\ell = |\up| + \ell_1 + 2\ell_2.
\end{equation}
Let $\sN(\up, \ell_1, \ell_2)$ be the number of base configurations that have a given type.  There are
\begin{equation}
 \frac{(2\ell)!}{|\up|! \ell_1! \ell_2! 2^{\ell_2}}
\end{equation}
ways to allot the $2\ell$ indices of $\um, \un$ to belong to $\up$, the $\ell_1$ singletons or $\ell_2$ doubles, and $\ll(N')^{\ell_1+\ell_2}$ ways to place the indices in $[N']$ once they have been so arranged, so that
\begin{equation}
 \sN(\up, \ell_1, \ell_2) \ll \frac{(2\ell)!}{|\up|!\ell_1! \ell_2! 2^{\ell_2}}(N')^{\ell_1+\ell_2}.
\end{equation}
Given $\um, \un$ of type $(\up, \ell_1, \ell_2)$,
\begin{align}
 &\E_{\eta^{\otimes N'}}\left[Me_{\ha}( \ouw_0)
 w_{m_1}^{(1)}\cdots w_{m_\ell}^{(1)}
w_{n_1}^{(2)}\cdots w_{n_\ell}^{(2)}\right] \\&\notag\leq \exp(-c_2 N' \|\ha\|^2) O(1)^{\ell_1 + \ell_2} \|\ha\|^{\omega(M)-|\up|+\ell_1} .
\end{align}
Indicating restriction of $\um, \un$ to base configurations with a $\prime$,
\begin{align}
& \sum_{\ell=0}^L \frac{(2\pi |\xi|)^\ell}{2^\ell \ell!}{\sum_{\um, \un \in
[N']^\ell}}'\left| \E_{\eta^{\otimes N'}}\left[Me_{\ha}( \ouw_0)
 w_{m_1}^{(1)}\cdots w_{m_\ell}^{(1)}
w_{n_1}^{(2)}\cdots w_{n_\ell}^{(2)}\right]\right|\\\notag
& \ll \exp(-c_2 N' \|\ha\|^2)\\\notag&\times\sum_{\up} \sum_{\substack{\ell_1, \ell_2 = 0\\ |\up| + \ell_1 \text{ even}}}^\infty \frac{ (\pi |\xi|)^{\frac{|\up| + \ell_1 +2\ell_2}{2}}\|\ha\|^{\omega(M)-|\up|+\ell_1} (|\up| + \ell_1 +2\ell_2)!}{\left(\frac{|\up| + \ell_1 +2\ell_2}{2} \right)! \ell_1!\ell_2! 2^{\ell_2}}O(N')^{\ell_1 + \ell_2}.
\end{align}
Bound
\begin{align}
 \frac{(|\up| + \ell_1 + 2\ell_2)!}{|\up|! \left(\frac{|\up| + \ell_1 + 2\ell_2}{2} \right)! \ell_1! \ell_2!}\leq 4^{|\up| + \ell_1 +2\ell_2}\frac{\ell_2!}{\left(\frac{|\up| + \ell_1 + 2\ell_2}{2} \right)!} \leq\frac{4^{|\up| + \ell_1 +2\ell_2}}{\left(\frac{ |\up|+\ell_1 }{2} \right)!}.
\end{align}
If the constant $c_1$ in (\ref{def_N_prime}) is chosen sufficiently small, then the sum over $\ell_2$ converges to a bounded quantity and the sum over $\ell_1$ is bounded by 
\begin{equation}\ll\exp\left(O(1) \|\ha\|^2 |\xi| (N')^2 \right).\end{equation} Since  $|\xi|N' \leq c_1$, if $c_1$ is sufficiently small this obtains a bound,
\begin{align}
& \ll \exp\left(-\frac{c_2}{2} N' \|\ha\|^2\right) \left(\|\ha\|^{\omega(m)} + \|\ha\|^{\omega(m)-1}\sqrt{|\xi|} + \cdots + |\xi|^{\frac{\omega(m)}{2}} \right).
\end{align}

This bound with $\alpha$ in place of $\ha$ obtains (\ref{monomial_bound}) for the dominant terms, and hence bounds the dominant terms unless $|\xi| \gg \frac{1}{N}$.  In the remaining case, the bound is acceptable unless $\|\ha\| < c_3 \|\alpha\|$ for a small
constant $c_3>0$.  In this case one obtains $\|\xi \ouw_t\| \gg
\|\alpha\|$.  Since $\ouw_t$ is a Gaussian with variance of order $N-N'<N$, the event $\|\xi \ouw_t\| \gg
\|\alpha\|$
occurs with $\uw_t$-probability, for some $c_4>0$, $\ll \exp\left( -c_4
\frac{\|\alpha\|^2}{N\xi^2} \right)$, which is again satisfactory.

To obtain a bound for all configurations from the bound for base ones, configurations with $|\sO| = \ell_1$, $|\sE| = \ell_2$ may be enumerated by adding a number $k$ of double indices to an existing base configuration.  There are $O(L)^k$ ways of choosing the indices where the new doubles will be added, $O(L)^{2k}$ ways of inserting the indices into the list $\um, \un$, and the new indices make a gain in the calculated moments of $O(L)^{k}$.  Meanwhile, a factor of $|\xi|^k$ is saved in the outer sum.  Recall $L \leq N^{2 \epsilon}$.  If $\epsilon <\frac{1}{8}$ then the sum over $k$ is $O(1)$ for all $N$ sufficiently large.

\end{proof}

\begin{proof}[Proof of Theorem \ref{zero_return_theorem}]
Combining Lemmas \ref{large_frequency_xi_lemma} and \ref{large_freq_abel_lemma}
 obtains, for any $A>0, 0< \epsilon< 
\frac{1}{4}$, for some $c
> 0$
\begin{align}\notag
 P_N(n_1, n_2, n_3) + O_A(N^{-A})
  =&\iint_{\substack{\|\alpha\| \leq
N^{\epsilon-\frac{1}{2}} \\ |\xi| \leq
\frac{\log N}{N}}}e_{\alpha}\left((n_1, n_2)^t \right)e_{\xi}\left(n_3 - \frac{n_1n_2}{2}  \right)  
\\&\qquad\times
\left[I\left( \sqrt{N}\sigma \alpha,  N\delta \xi\right)+O\left(E
\right)\right]  d\alpha d\xi,
\end{align}
where the error term $E$ satisfies the estimates of Lemmas
\ref{Gaussian_replacement_lemma}, \ref{large_xi_error_lemma} and \ref{large_alpha_error_lemma}.  Over the range of integration the error integrates to
$O\left(N^{-\frac{5}{2}}\right)$.

 Making a change of variables and extending the integral
to $\bR^3$ obtains
\begin{align}\notag
 P_N(n_1, n_2, n_3) +O\left(N^{-\frac{5}{2}} \right)=&
\frac{1}{  \delta^2 
N^2}\int_{\bR^3}e_{\alpha}\left(\sigma^{-1}  
\left[ \frac{(n_1, n_2)^t
}{\sqrt{N}}\right]\right)\\& \times e_{\xi} 
\left(\frac{1}{\delta}\left(
 \frac{n_3 - \frac{n_1n_2}{2}}{N}  
\right)\right)
I\left(\alpha,   \xi\right)d\alpha  d\xi.
\end{align}
The right hand side is the Gaussian density of the limit theorem.
To obtain the return probability to 0, use $\delta^2 =
\frac{4}{25}$ and
$\int_{\bR^3}I(\alpha,  \xi)d\alpha  d\xi = \frac{1}{4}$ in \begin{align}P_N(0,0,0) &= \frac{1}{\delta^2N^2}\int_{\bR^3}I(\alpha, \xi)d\alpha d\xi + O\left(N^{-\frac{5}{2}}\right) \\\notag&= \frac{25}{16 N^2} + O\left(N^{-\frac{5}{2}}\right).\end{align} 

\end{proof}

\section{Proof of Theorem \ref{local_limit_theorem},  Cram\'{e}r case}

Theorem \ref{local_limit_theorem} treats measures for which the abelianized walk is non-lattice.    In this case the fibered distribution is also 
dense in $\bR$, and when the abelianized distribution satisfies  the Cram\'{e}r condition, the fibered distribution does, also.  We assume the Cram\'{e}r condition in this section and treat the general case in the next section.  In this case, after making an arbitrary translation on the left and right, the test functions may be taken to be of the form
\begin{equation}
 f([x,y,z]) = F\left(x-x_0, y-y_0, z-\frac{xy}{2} -Ax - By - z_0 \right)
\end{equation}
where $F$ is a Lipschitz function of compact support and $x_0, y_0,
z_0, A, B$ are real parameters. 

Let $\rho$ be a smooth, compactly supported bump function on $\bR^3$, for $t > 0$, $\rho_t(x) = t^3 \rho(tx)$ and $F_t = F * \rho_t$ the convolution on $\bR^3$.  Since $F$ is Lipschitz, $\|F - F_t\|_\infty \ll \frac{1}{t}$ as $t \to \infty$.  Set
\begin{equation}
f_t([x,y,z]) = F_t\left(x-x_0, y-y_0, z-\frac{xy}{2} -Ax - By - z_0 \right). 
\end{equation}
Choosing $t = t(N) = N^{\frac{5}{2}}$,  
\begin{align} 
&\langle f, \mu^{*N}\rangle=  O\left(N^{-\frac{5}{2}}\right) +\int_{(\alpha, \xi) \in \bR^3} \hat{F}_t(\alpha, \xi)\\\notag & \times \left\langle e_{\alpha}\left((x-x_0, y-y_0)^t\right)e_{\xi}\left(z - \frac{xy}{2}-Ax - By -z_0\right), \mu^{*N} \right\rangle d\alpha d\xi. 
\end{align}
Using the decay of the Fourier transform of the bump function $\rho_t$, truncate the integral to $\|\alpha\|, |\xi| = O(N^{O(1)})$ with admissible error. 

Apply the multiplication rule (\ref{multiplication_rule}) to write the central coordinate of a product of group elements $\uw$ as 
\begin{equation}
\tilde{z} = z-\frac{xy}{2} = H^*(\uw) + \overline{\underline{\tilde{w}}}^{(3)}.
\end{equation}
The mean of $\overline{\underline{\tilde{w}}}^{(3)}$ is  $N\overline{\tilde{z}}$.  Let $\overline{\underline{\tilde{w}}}^{(3)}_0 = \overline{\underline{\tilde{w}}}^{(3)} - N\overline{\tilde{z}}$.   
Let $\tilde{\alpha} = \alpha - \xi\cdot(A,B)^t$. Thus
\begin{align}
&\left\langle e_{\alpha}\left((x-x_0, y-y_0)^t\right)e_{\xi}\left(z - \frac{xy}{2}-Ax - By -z_0\right), \mu^{*N} \right\rangle=\\&\notag
e_{\alpha}\left(-(x_0,y_0)^t\right)e_{\xi}\left(N\overline{\tilde{z}}-z_0\right)\int_{\bH(\bR)^N}e_{\tilde{\alpha}}\left(\ouw_{\ab}\right)e_{\xi}\left(H^*(\uw) + \overline{\underline{\tilde{w}}}^{(3)}_0\right)d\mu^{\otimes N}.
\end{align}

The argument of Lemma \ref{large_frequency_xi_lemma} applies as before to truncate to $|\xi| = O\left(\frac{\log N}{N}\right)$.  This uses Lemma \ref{H_star_decay} in the case $|\xi| = O(1)$ and Lemma \ref{H_star_cramer} in the case $|\xi| \gg 1$. The argument of Lemma \ref{large_freq_abel_lemma} applies as before to truncate to $\left\|\tilde{\alpha}\right\| \ll N^{-\frac{1}{2}+\epsilon}$.  

A small modification is needed to the application of Lemma \ref{Gaussian_replacement_lemma} which we now describe.  
Here one can now include in the measure $\mu_{\ab}$ a third dimension corresponding the $\tilde{z}-\overline{\tilde{z}}$, and make $\eta$ a 3 dimensional Gaussian with the same covariance matrix.  The Gaussian replacement scheme goes through essentially unchanged, the addition of the third coordinate evaluated at the small frequency $\xi$ making a negligible change; these terms do not need to be included in $T_j(\tilde{\alpha}, \xi, \uw)$.  The main term becomes
\begin{equation}
\E_{\eta^{\otimes N}}\left[e_{\tilde{\alpha}}(\ouw_{\ab})e_\xi(H^*(\uw))e_{\xi}\left(\overline{\underline{\tilde{w}}}^{(3)}_0 \right) \right].
\end{equation}
After a linear change of coordinates, the third coordinate is independent of the first two and $\tilde{\alpha}$ is mapped to $\alpha' = \tilde{\alpha} + O(\xi)$.  Hence
\begin{align}
&\E_{\eta^{\otimes N}}\left[e_{\tilde{\alpha}}(\ouw_{\ab})e_\xi(H^*(\uw))e_{\xi}\left(\overline{\underline{\tilde{w}}}^{(3)}_0 \right) \right]\\\notag& = \left(1 + O(\xi^2 N)\right)\E_{\eta^{\otimes N}}\left[e_{\alpha'}(\ouw_{\ab})e_\xi(H^*(\uw)) \right].
\end{align}
Note that, since $\|\alpha'\|^2 = \|\tilde{\alpha}\|^2 + O(\|\tilde{\alpha}\||\xi|) + O(|\xi|^2)$, 
\begin{align}
 I\left(N^{\frac{1}{2}}\sigma\alpha', N\delta\xi\right) &=\frac{\exp\left(-\frac{2\pi \|\alpha'\|^2}{\delta \xi \coth N\delta \pi \xi} \right)}{\cosh N\delta \pi \xi}\\
 \notag &= I\left(N^{\frac{1}{2}}\sigma\tilde{\alpha}, N\delta\xi\right) \left(1 + O(\|\tilde{\alpha}\| + |\xi|)\right).
\end{align}

In the error term, 
\begin{equation}
\E_{\eta^{\otimes N}}\left[e_{\tilde{\alpha}}(\ouw_{\ab})e_\xi(H^*(\uw))e_{\xi}\left(\overline{\underline{\tilde{w}}}^{(3)}_0 \right)\sum_j T_j(\tilde{\alpha}, \xi, \uw) \right],
\end{equation}
the factor of $e_{\xi}\left(\overline{\underline{\tilde{w}}}^{(3)}_0 \right)$ may be removed by Taylor expanding to degree 1, so that this part of the argument is unchanged.

To complete the argument, integrate as before
\begin{align}
\left \langle f, \mu^{*N}\right \rangle &=  \int_{\substack{\|\tilde{\alpha}\| \ll N^{-\frac{1}{2}+\epsilon}\\ |\xi| \ll \frac{\log N}{N}}}\hat{F}_t(\alpha,\xi)e_\alpha(-(x_0,y_0)^t)e_\xi(N\overline{\tilde{z}} - z_0)\\\notag &\times \left[I\left(N^{\frac{1}{2}}\sigma\alpha', N\delta\xi\right)+ O(E) \right] d\alpha d\xi + O\left(N^{-\frac{5}{2}}\right).
\end{align}
The argument is now completed essentially as before, to find
\begin{equation}
\left \langle f, \mu^{*N}\right\rangle = \left \langle f_t, d_{\sqrt{N}} \nu\right\rangle + O\left(N^{-\frac{5}{2}} \right) =  \left \langle f, d_{\sqrt{N}} \nu\right\rangle + O\left(N^{-\frac{5}{2}} \right).
\end{equation}
\section{Proof of Theorem \ref{local_limit_theorem}, general case}
  We now consider the case in which $\mu_{\ab}$ does not necessarily satisfy a Cram\'{e}r condition.  In this section the test functions take the form
  \begin{equation}
  f([x,y,z])= F\left(x-x_0, y-y_0, z - \frac{xy}{2} - Ax - By -z_0 \right)
  \end{equation}
  with $F$ continuous and of compact support.  Since we ask only for an asymptotic, it suffices by Selberg's theory of band-limited majorants and minorants \cite{GMM15} to assume that $F$ takes the form
\begin{equation}
 F([x,y,z]) = \phi_{\ab}(x,y)\phi_3(z)
\end{equation}
with $\phi_{\ab}$ and $\phi_3$ functions  with Fourier transform of compact support.  In this case, writing $\tilde{\alpha} = \alpha - \xi \cdot(A,B)^t$,
\begin{align}
\left\langle f, \mu^{*N}\right\rangle = &\int_{\|\alpha\|, |\xi| = O(1)} \hat{\phi}_{\ab}(\alpha)\hat{\phi}_3(\xi) e_{\alpha}\left(-(x_0,y_0)^t\right)e_{\xi}\left(N \overline{\tilde{z}}-z_0\right)\\
&\notag \times\int_{\bH(\bR)^N}e_{\tilde{\alpha}}\left(\ouw_{\ab}\right)e_{\xi}\left(H^*(\uw) + \overline{\underline{\tilde{w}}}_0^{(3)}\right)d\mu^{\otimes N}.
\end{align}

Argue as in Lemma \ref{large_frequency_xi_lemma} to truncate to $|\xi|\ll \frac{\log N}{N}$.  Since $A$ and $B$ are unconstrained, a further difficulty is encountered in applying Lemma \ref{large_freq_abel_lemma} to truncate $\alpha$.  
Let $\epsilon > 0$.  For $|\xi| \ll \frac{\log N}{N}$ and $\tilde{\alpha} \not \in \Spec_{N^{-1+\epsilon}}(\mu_{\ab})$, Lemma \ref{large_freq_abel_lemma} demonstrates that the integral over $\bH(\bR)^N$ is, for any $A>0$, $O_A(N^{-A})$.
The following modification of Lemma \ref{Gaussian_replacement_lemma} permits an asymptotic evaluation of 
\begin{equation}
\E_{\bU_N}\left[e_{\tilde{\alpha}}(\ouw_{\ab})e_{\xi}\left(H^*(\uw) + \overline{\underline{\tilde{w}}}_0^{(3)}\right)\right]
\end{equation}
at points $\tilde{\alpha}$ in the large spectrum of $\mu_{\ab}$, $\Spec_{N^{-1+\epsilon}}(\mu_{\ab})$. 
\begin{lemma}\label{large_spectrum_char_fun_approx_lemma}Let $\mu_{\ab}$ have covariance matrix $\sigma^2$ and set $\delta = \det(\sigma)$.  
	Let $0 < \epsilon < \frac{1}{4}$, $\vartheta = N^{-1+\epsilon}$, let $\tilde{\alpha} \in \Spec_{N^{-1+\epsilon}}(\mu_{\ab})$, $|\xi| \ll \frac{\log N}{N}$, and let $\alpha_0 \in \sM_{\vartheta}(\mu_{\ab})$ satisfy $\left\|\tilde{\alpha}-\alpha_0\right\| \ll \sqrt{\vartheta}$. Then
	\begin{align}
	&\E_{\bU_N}\left[e_{\tilde{\alpha}}(\ouw_{\ab})e_\xi\left(H^*(\uw) + \overline{\underline{\tilde{w}}}_0^{(3)}\right) \right]\\ \notag & = O\left(N^{-\frac{1}{2}+O(\epsilon)}\right) + \hat{\mu}_{\ab}(\alpha_0)^N  I\left(N^{\frac{1}{2}}\sigma (\tilde{\alpha}-\alpha_0), N\delta \xi\right).
	\end{align}
\end{lemma}
\begin{proof}Write $e_\xi\left(\overline{\underline{\tilde{w}}}^{(3)}_0\right) = 1 + O\left(|\xi|\overline{\underline{\tilde{w}}}^{(3)}_0\right).$  Since
	\begin{equation}
	\E_{\bU_N}\left[|\xi|\left|\overline{\underline{\tilde{w}}}^{(3)}_0 \right| \right] = O\left(N^{-\frac{1}{2}+\epsilon}\right)
	\end{equation}
	it suffices to prove that 
	\begin{align}
	&\E_{\bU_N}\left[e_{\tilde{\alpha}}(\ouw_{\ab})e_\xi(H^*(\uw))\right] \\\notag &=\hat{\mu}_{\ab}(\alpha_0)^N I\left(N^{\frac{1}{2}}\sigma (\tilde{\alpha}-\alpha_0), N\delta \xi\right) + O\left(N^{-\frac{1}{2}+O(\epsilon)}\right). 
	\end{align}
	
	Let $\eta = \eta(0, \sigma)$ be a centered two dimensional Gaussian with covariance equal to that of $\mu_{\ab}$.  Let $X$ be distributed according to $\eta$. Set  $\alpha = \tilde{\alpha}-\alpha_0$.
	By Lemma \ref{Taylor_expansion_near_local_max},
	\begin{equation}
	\hat{\mu}_{\ab}(\alpha_0 + \alpha) = \hat{\mu}_{\ab}(\alpha_0)\E_{\eta}[e_\alpha(X)] + O(\vartheta \|\alpha\|  + \|\alpha\|^3).
	\end{equation}
	In analogy with Lemma \ref{Gaussian_replacement_lemma}, define
	\begin{equation}
	\alpha_j(\uw) = \alpha+ \frac{\xi}{2} \left[\sum_{i \neq
		j}(-1)^{\delta(i<j)}w_i^{(2)}, \sum_{i \neq
		j}(-1)^{\delta(i>j)}w_i^{(1)} \right]^t.
	\end{equation}
	Set, for $1 \leq k \leq N$,
	$\mu_k = \mu_{\ab}^{\otimes k} \otimes \eta^{\otimes (N-k)}$. Also set $\ouw_{\ab, k} = \sum_{j=1}^k w_{j,\ab}$ and 
	\begin{equation}
	S_k = \E_{\mu_k}\left[e_\alpha(\ouw_{\ab})e_{\alpha_0}(\ouw_{\ab,k})e_\xi(H^*(\uw))\right]
	\end{equation}
	so that $S_N = \E_{\bU_N}\left[e_{\tilde{\alpha}}(\ouw_{\ab})e_\xi(H^*(\uw))\right]$ and 
	\begin{equation}
	S_0 = I\left(\sqrt{N}\sigma \alpha, N\delta \xi; N\right) =  I\left(\sqrt{N}\sigma \alpha, N\delta \xi\right) + O\left(N^{-1 +O(\epsilon)}\right).
	 \end{equation}
	
	Holding all but the $k$th variable fixed obtains 
	\begin{align}
	|S_k - \hat{\mu}_{\ab}(\alpha_0)S_{k-1}| &\ll  \E_{\mu_{\ab}^{\otimes(k-1)}\otimes \eta^{\otimes (N-k)}}\left[\vartheta\|\alpha_k(\uw)\| + \|\alpha_k(\uw)\|^3\right]\\ \notag & = O\left(N^{-\frac{3}{2}+O(\epsilon)}\right).
	\end{align}
	The proof is now complete, since
	\begin{align}
	 \left|S_N - \hat{\mu}_{\ab}(\alpha_0)^N S_0\right| & \leq \sum_{k = 0}^{N-1} \left|\hat{\mu}_{\ab}(\alpha_0)^k S_{N-k} -\hat{\mu}_{\ab}(\alpha_0)^{k+1} S_{N-k-1}  \right|\\
	 &\notag \leq \sum_{k=0}^{N-1} \left|S_{N-k} -  \hat{\mu}_{\ab}(\alpha_0) S_{N-k-1}\right| \ll N^{-\frac{1}{2}+O(\epsilon)}.
	\end{align}

\end{proof}

\begin{proof}[Proof of Theorem \ref{local_limit_theorem}, general case] 
Let $0 < \epsilon < \frac{1}{4}$ and let $\vartheta = N^{-1+\epsilon}$.  By Lemma \ref{large_spectrum_structure_lemma} there is a constant $c_1 = c_1(\mu)$ such that 
\[
\Spec_\vartheta(\mu_{\ab}) \subset \bigcup_{\alpha_0 \in \sM_\vartheta}B_{c_1\vartheta^{\frac{1}{2}}}(\alpha_0). 
\]

Let $\supp \hat{\phi}_{\ab} \subset B_R(0)$.  Define the set  
\begin{equation}
\Xi_{\good} = \left\{\xi: |\xi| \ll \frac{\log N}{N}, \;-\xi \cdot (A,B)^t \in \bigcup_{\alpha_0 \in \sM_\vartheta}B_{R+c_1\vartheta^{\frac{1}{2}}}(\alpha_0)\right\}. 
\end{equation}
Assume that $N$ is sufficiently large so that  if $\alpha_0, \alpha_1$ are distinct points of $\sM_\vartheta$ then $\|\alpha_0 -\alpha_1\| \geq 2 (R+c_1\vartheta^{\frac{1}{2}}).$  Given $\xi \in \Xi_{\good}$ let $\alpha_0(\xi)$ be the nearest point to $-\xi \cdot (A,B)^t$ in $\sM_{\vartheta}$.  Also, define 
\begin{equation}
A_\xi = \left\{\alpha \in B_{R}(0): \tilde{\alpha} = \alpha - \xi \cdot (A,B)^t \in \Spec_{\vartheta}(\mu_{\ab})\right\}. 
\end{equation}
By Lemma \ref{large_spectrum_structure_lemma}, $|A_{\xi}| \ll N^{-1 + \epsilon}$.  

In the evaluation from above,
\begin{align}
&\left \langle f, \mu^{*N}\right \rangle + O_A(N^{-A}) \\\notag &= \int_{\xi \in \Xi_{\good}} \int_{\alpha: \tilde{\alpha} \in \Spec_\vartheta(\mu_{\ab})}\hat{\phi}_{\ab}(\alpha)\hat{\phi}_3(\xi)e_\alpha\left(-(x_0,y_0)^t\right)e_\xi\left(N\overline{\tilde{z}}-z_0\right) 
\\&\notag \times \E_{\mu^{\otimes  N}}\left[e_{\tilde{\alpha}}(\ouw_{\ab})e_\xi(H^*(\uw)) \right]d\alpha d\xi 
\end{align}
insert the asymptotic formula for the expectation from Lemma \ref{large_spectrum_char_fun_approx_lemma}.  The error term here is bounded by
\begin{align}
&\int_{\xi: |\xi|\ll \frac{\log N}{N}}\int_{\substack{\alpha \in B_{R}(0),\\ \tilde{\alpha} \in \Spec_{\vartheta}(\mu_{\ab})}}N^{-\frac{1}{2} + O(\epsilon)}d\alpha d\xi
\\ &\notag \leq \int_{\xi \in \Xi_{\good}}N^{-\frac{1}{2}+O(\epsilon)} |A_\xi|d\xi = O\left(N^{-\frac{5}{2}+O(\epsilon)}\right).
\end{align}

This leaves the main term
\begin{align}
&\notag\int_{\xi \in \Xi_{\good}}\int_{\alpha: \tilde{\alpha} \in \Spec_\vartheta(\mu_{ab})}\hat{\phi}_{\ab}(\alpha)\hat{\phi}_3(\xi)e_\alpha\left(-(x_0,y_0)^t\right)e_\xi\left(N\overline{\tilde{z}}-z_0\right)\\& \times I\left(\sqrt{N}\sigma (\tilde{\alpha}-\alpha_0(\xi)), N\delta \xi\right)d\alpha d\xi. 
\end{align}
The contribution from the part of this integral where $\alpha_0(\xi) = 0$ contributes $\langle f, d_{\sqrt{N}} \nu\rangle + O(N^{-A})$, by extending the integral with the same integrand to all of $\bR^3$, so it remains to bound the contribution from $\xi$ for which $\alpha_0(\xi) \neq 0$.

For a fixed $\xi \in \Xi_{\good}$, the formula 
\begin{equation}
I(\alpha, \xi) = \frac{\exp\left(-\frac{2\pi \|\alpha\|^2}{\xi\coth \xi}\right)}{\cosh \pi \xi}
\end{equation}
gives that integration in $\alpha$ is bounded absolutely by
\begin{equation}
\int_{\alpha}I\left(\sqrt{N}\sigma (\tilde{\alpha}-\alpha_0(\xi)), N\delta \xi\right)d\alpha\ll \frac{\max\left(\frac{1}{N}, |\xi|\right)}{\cosh N\pi \delta \xi}.
\end{equation}
The contribution of $\xi \in \Xi_{\good}$ for which $\alpha_0(\xi) \neq 0$ is bounded by
\begin{align}
\int_{\xi \in \Xi_{\good}, \alpha_0(\xi)\neq 0}\frac{\max(\frac{1}{N}, |\xi|)}{\cosh N\pi \delta\xi}d\xi. 
\end{align}
Since the $\alpha_0(\xi)$ are $\sF(\vartheta)$ spaced, this is bounded by
\begin{equation}
\ll\frac{1}{\sF(\vartheta)}\int_0^\infty \frac{\max\left(\frac{1}{N}, |\xi|\right)}{\cosh N\pi \delta\xi}d\xi = o\left(\frac{1}{N^2}\right).
\end{equation}
\end{proof}

\section{Random walk on $N_n(\zed)$, proof of Theorem \ref{N_n_theorem}}

The case $n=2$ is classical so assume $n \geq 3$. 

Let $M:\zed^{n-1} \to N_n(\zed)$ be the map
\begin{equation}
M: \zed^{n-1}\ni v = \left( \begin{array}{c} v^{(1)}\\ v^{(2)}\\ \vdots \\
v^{(n-1)} \end{array}\right)  \mapsto \left(\begin{array}{ccccc}
1 & v^{(1)} & 0 &\cdots & 0
\\ 0 & 1     & v^{(2)} & 0 & \vdots
\\ \vdots  &  \ddots     &\ddots   &\ddots & 0
\\   &   0    &     0    & 1 & v^{(n-1)}\\
 0    &  \cdots     &         & 0  &   1\end{array} \right)
\end{equation}
Recall that, given $m \in N_n$ we write $Z(m)$ for the upper right corner.
Given sequence of vectors $\uv = \{v_i\}_{i=1}^N \in \left(\zed^{n-1}\right)^N$
the central coordinate satisfies the product rule
\begin{equation}
 Z\left( \prod_{i=1}^N M(v_i)\right) = \sum_{1 \leq i_1 < i_2 < ...< i_{n-1}
\leq N} v_{i_1}^{(1)} v_{i_2}^{(2)} \cdots v_{i_{n-1}}^{(n-1)}.
\end{equation}
Write
\begin{equation}
 Z_{n}^N = \sum_{1 \leq i_1 < i_2 < ... < i_{n-1}\leq N} e_{i_1}^{(1)} \otimes
\cdots \otimes e_{i_{n-1}}^{(n-1)}
\end{equation}
for the corresponding tensor.  $Z_{n, \mu}^N$ denotes the measure on $\zed$
obtained by pushing forward measure $\mu$ on $\zed^{n-1}$ via $M$ to measure
$\tilde{\mu}$ on $N_n(\zed)$, then obtaining 
$\langle Z, \tilde{\mu}^{*N}\rangle$.
Equivalently, $Z_{n, \mu}^N$ is the distribution of $Z_n^N$ evaluated on $N$
vectors $v_i$  drawn i.i.d. from $\mu$.

Given a probability measure $\nu$ on $\zed$ and prime $p$, Cauchy-Schwarz and 
Plancherel give
\begin{equation}\label{upper_bound_lemma}
 \sum_{x \bmod p} \left|\nu(x \bmod p) - \frac{1}{p}\right| \leq \left(\sum_{0 \not
\equiv \xi  \bmod p} \left|\hat{\nu}\left(\frac{
\xi}{p}\right)\right|^2\right)^{\frac{1}{2}}
\end{equation}
where
\begin{align}
 &\hat{\nu}(\alpha) =
 \sum_{n \in \zed} e_{-\alpha}(n) \nu(n).
\end{align}
Theorem \ref{N_n_theorem} thus reduces to the following estimate on the
characteristic function of $Z_{n, \mu}^N$.
\begin{proposition}\label{char_fun_proposition} Let $n \geq 3$ and let
 $\mu$ be a measure on $\zed^{n-1}$ satisfying the same conditions as in Theorem
\ref{N_n_theorem}.
 There exists constant $C>0$ such that  for all $N>0$ and all $0 <
|\xi| \leq \frac{1}{2}$,
 \begin{align}
 \left|\hat{Z}_{n, \mu}^{N} \left( \xi\right)\right|  &\ll \exp\left(-C N
|\xi|^\frac{2}{n-1} \right).
 \end{align}
\end{proposition}

\begin{proof}[Deduction of Theorem \ref{N_n_theorem}]
Recall $N = c p^{\frac{2}{n-1}}$ and let $c \geq 1$. Apply the upper bound of 
Proposition
\ref{char_fun_proposition}. By (\ref{upper_bound_lemma}),
\begin{align}
 \left(\sum_{x \bmod p} \left|Z_{n,\mu}^N(x) - \frac{1}{p}\right| \right)^2
&\leq \sum_{\substack{\xi \in \zed\\ 0 < |\xi| < \frac{p}{2}}}\left|\hat{Z}_{n,
\mu}^{N}\left(\frac{ \xi}{p}\right) \right|^2 \\ \notag & \ll  \sum_{0 < |\xi| <
\frac{p}{2}}  \exp\left(-Cc |\xi|^{\frac{2}{n-1}} \right)\\ \notag
 &\ll \exp\left(-Cc\right).
\end{align}
\end{proof}

\subsection{Proof of Proposition \ref{char_fun_proposition}} Let $C_2^{n-2}$ act on
blocks of vectors of length $k 2^{n-2}$ with the $j$th factor from $C_2^{n-2}$,
$j \geq 1$ switching the relative order of the first $k2^{j-1}$ and second
$k2^{j-1}$ indices.  Thus, for instance, in case $n = 5$, if each of $x_1, ..., 
x_8$ represents a block of $k$ consecutive indices and $\ux = 
x_1x_2x_3x_4x_5x_6x_7x_8$,
\begin{align}
\notag \tau_2 \ux &= x_3x_4x_1x_2 x_5x_6x_7x_8 \\
 \tau_1 \tau_3 \ux = \tau_3 \tau_1\ux &= x_5x_6x_7x_8 x_2x_1x_3x_4\\
\notag \tau_1\tau_2\tau_3 \ux &= x_5x_6x_7x_8x_3x_4x_2x_1.
\end{align}
For $k \geq 1$ set $N' = \left\lfloor \frac{N}{k 2^{n-2}} \right \rfloor$
and let $G_k = (C_2^{n-2})^{N'}$.  $G_k$ acts on sequences of length $N$ with, 
for $j \geq 1$, the $j$th factor of $G_k$ acting on the contiguous subsequence 
of indices of length $k2^{n-2}$ ending at $jk2^{n-2}$.  
For fixed $k$ and fixed $\uw \in W_N =(\supp \mu)^N$,  let
\begin{equation}
 Z_k(\uw) = \E_{\utau \in G_k} \left[\delta_{Z_n^N(\utau \cdot \uw)} \right].
\end{equation}
Continue to abbreviate $\bU_N = \mu^{\otimes N}$. For any $k$,
\begin{equation}
Z_{n, \mu}^N = \E_{\bU_N}\left[Z_k(\uw)\right].
\end{equation}

We introduce a second, dual action of $G_k$  on a
linear dual space.  Let 
\begin{equation}
 \sI_{N} = \left\{\ui = (i_1, i_2, ..., i_{n-1}):  1 \leq i_1 < i_2 <
... < i_{n-1} \leq N\right\}.
\end{equation}
 Given $\ui \in \sI_{N}$ and $k \geq 1$, let $\fS_{\ui,
k} \subset \fS_{n-1}$ be the subset of permutations $\fS_{\ui, k} =
\{\sigma_{\utau, \ui}: \utau \in G_k\}$ where
\begin{equation}
 \forall 1 \leq j \leq n-1,  \quad \sigma_{\utau, \ui}(j) =\#\{1 \leq k 
\leq
n-1: \utau(i_k)\leq \utau(i_j)\}.
\end{equation}
That is, $\sigma_{\utau, \ui}(j)$ is the relative position of $\utau \cdot 
i_j$ when
$\utau \cdot \ui$ is sorted to be in increasing order. Put another way, suppose 
$\utau$ maps $i_1 < \cdots < i_{n-1}$ to  $j_1 < \cdots < j_{n-1}$ in some 
order (and vice 
versa, $\utau$ is an involution) and calculate
\begin{align}\notag
e_{j_1}^{(1)} \otimes \cdots \otimes e_{j_{n-1}}^{(n-1)}(\utau \cdot \uw)
&= e_{\utau \cdot j_1}^{(1)} \otimes \cdots \otimes e_{\utau \cdot 
j_{n-1}}^{(n-1)} (\uw)\\
&= e_{i_{\sigma^{-1}(1)}}^{(1)} \otimes \cdots \otimes 
e_{i_{\sigma^{-1}(n-1)}}^{(n-1)}(\uw)\\
\notag &= e_{i_1}^{(\sigma(1))} \otimes \cdots \otimes 
e_{i_{n-1}}^{(\sigma(n-1))}(\uw),
\end{align}
where $\sigma_{\utau, \ui}$ is abbreviated $\sigma$.

Let
\begin{equation}
X_{N,k} = \left\{e_{i_1}^{(\sigma(1))} \otimes \cdots \otimes
e_{i_{n-1}}^{(\sigma(n-1))}: \ui \in \sI_{N}, \sigma \in \fS_{\ui, k}\right\}.
\end{equation}
The action of $\utau \in G_k$ is defined on a representative set within
$X_{N,k}$  by, for each $\ui \in \sI_N$,
\begin{equation}\label{group_action_def}
 \utau \cdot \left(e_{i_1}^{(1)} \otimes \cdots \otimes
e_{i_{n-1}}^{(n-1)}\right) = e_{i_1}^{(\sigma_{\utau, \ui}(1))} \otimes \cdots
\otimes e_{i_{n-1}}^{(\sigma_{\utau, \ui}(n-1))}.
\end{equation}
The following lemma justifies that this definition extends to a unique group 
action of $G_k$ on all of $X_{N,k}$.
\begin{lemma}
 Let $\utau, \utau' \in G_k$ and $\ui \in \sI_N$ satisfy $\sigma_{\utau, \ui} =
\sigma_{\utau', \ui}$.  Then for any $\utau'' \in G_k$, $\sigma_{\utau +\utau'',
\ui} = \sigma_{\utau'+\utau'', \ui}$.  In particular, (\ref{group_action_def})
extends to a unique group action on $X_{N,k}$.
\end{lemma}

\begin{proof}
 This follows, since, for any $1 \leq i < j \leq N$ there is at most one factor
of $G = C_2^{n-2}$ in $G_k$, and one index $\ell$, $1 \leq \ell \leq n-2$ of $G$
which exchanges the order of $i$ and $j$. To define the group action in general,
given $\utau \in G_k$, $\ui \in \sI_N$ and $\sigma \in \fS_{\ui, k}$ choose any 
$\utau_0$
such that $\sigma = \sigma_{\utau_0, \ui}$.  Let $\sigma' = \sigma_{\utau_0 +
\utau, \ui}$.  Then
 \begin{equation}
  \utau \cdot e_{i_1}^{(\sigma(1))} \otimes \cdots \otimes
e_{i_{n-1}}^{(\sigma(n-1))} = e_{i_1}^{(\sigma'(1))} \otimes \cdots \otimes
e_{i_{n-1}}^{(\sigma'(n-1))}.
 \end{equation}
 The definition is clearly unique, since (\ref{group_action_def}) surjects on
$X_{N,k}$.
\end{proof}

The actions of $G_k$ on $W_N$ and on $X_{N,k}$, although not adjoint, are
compatible on $Z_n^N$, in the sense that for any $\utau$,
\begin{equation}
 (\utau \cdot Z_n^N)(\uw) = Z_n^N(\utau \cdot \uw)
\end{equation}
so that
\begin{equation}
Z_k(\uw) = \E_{\utau \in G_k} \left[\delta_{(\utau \cdot Z_n^N)(\uw)} \right].
\end{equation}

Note that in general, although $G_k$ is a product group, the separate
factors $\tau_i$ do not act independently in $\utau \cdot Z_n^N$.  For instance,
when $n=5$ and $k = 1$, the change in a tensor of type $e_1^{(1)} \otimes 
e_2^{(2)}\otimes
e_{2^{n-2}+1}^{(3)} \otimes e_{2^{n-2}+2}^{(4)}$ under the first factor of
$G_k$ depends upon whether the second factor has been applied.  Thus the
characteristic function
\begin{equation}
 \chi_k(\xi, \uw) = \E_{\utau \in G_k} \left[e_{-\xi}\left( Z_n^N(\utau 
\cdot \uw)\right)
\right]
\end{equation}
need not factor as a product.

A pleasant feature of the general case is that this difficulty is rectified by
estimating instead of $\chi_k(\xi, \uw)$, a function $F_k(\xi, \uw)$ which is
the result of applying the Gowers-Cauchy-Schwarz inequality to $\chi_k(\xi,
\uw)$. To describe this, write $G_k = (C_2^{n-2})^{N'} = (C_2^{N'})^{n-2}$, and 
thus 
\begin{equation}
 \E_{\utau \in G_k} [f(\utau)] = \E_{\utau_1 \in C_2^{N'}} \cdots
\E_{\utau_{n-2} \in C_2^{N'}} \left[f(\utau_1, ..., \utau_{n-2}) \right].
\end{equation}
Then, setting apart one expectation at a time and applying Cauchy-Schwarz, 
\begin{align}
 &\left|\chi_k(\xi, \uw)\right|^{2^{n-2}} \\
\notag &\leq \E_{\utau_1, \utau_1' \in C_2^{N'}} \cdots \E_{\utau_{n-2}, \utau_{n-2}'
\in C_2^{N'}} \left[e_{-\xi}\left( \sum_{S \subset [n-2]}(-1)^{n-2-|S|} \utau_S
\cdot \uw \right)\right]
 \\ \notag & = \E_{\utau, \utau' \in G_k} \left[e_{-\xi}\left( \sum_{S \subset
[n-2]}(-1)^{n-2-|S|} \utau_S \cdot \uw \right)\right]
 \\ \notag & =: F_k(\xi, \uw),
\end{align}
where
\begin{equation}
 \utau_S = (\utau_{S, 1}, ..., \utau_{S, n-2}), \qquad \utau_{S, i} = \left\{
\begin{array}{ccc} \utau_i && i \in S\\ \utau_i' && i \not \in
S\end{array}\right..
\end{equation}

\begin{lemma}\label{factorization_lemma}
 $F_k(\xi, \uw)$ factors as the product
\begin{align}
 F_k(\xi, \uw)&= \prod_{j=1}^{N'} \left(1 - \frac{1}{2^{n-2}} +
\frac{F_{k,j}(\xi, \uw)}{2^{n-2}} \right)
\end{align}
where $F_{k,j}(\xi, \uw)$ is a function of $\uw_{j,k} = (\omega_1, ...,
\omega_{2^{n-2}})$ with the 
\begin{equation}\omega_i = \sum_{(2^{n-2}(j-1) + i-1)k < \ell \leq
(2^{n-2}(j-1)+i)k} w_\ell\end{equation} the sum of consecutive blocks of length $k$ in
$\uw$.   Identify $C_2^{n-2}$ with $\{0,1\}^{n-2}$ and write $|\tau| =
\sum_{i=1}^{n-2} \one(\tau_i \neq 0)$.
Then
\begin{equation}
 F_{k,j}(\xi, \uw) = \E_{\tau \in C_2^{n-2}} \left[e_{-\xi}\left( 
\sum_{\tau'
\in C_2^{n-2}}(-1)^{|\tau'|}Z_{n}^{2^{n-2}}((\tau + \tau') \cdot \uw_{j,k})
\right) \right]
\end{equation}
with the action of $C_2^{n-2}$ on blocks of size 1 in $\uw_{j,k}$.
\end{lemma}

\begin{proof}
Consider for fixed $\utau, \utau' \in G_k$ the sum
\begin{equation}
 Z_n^N(\utau, \utau')(\uw) = \sum_{S \subset [n-2]} (-1)^{n-2-|S|} \utau_S \cdot
Z_n^N(\uw).
\end{equation}
After replacing $\uw$ with $\utau' \uw$ and $\utau$ with $\utau +\utau'$ it 
suffices to consider $\utau' = \id$. 

Consider the action of
\begin{equation}
\hat{\tau} = \sum_{S \subset [n-2]} (-1)^{n-2-|S|} \tau_S
\end{equation}
on a tensor
\begin{equation}
\be = e_{i_1}^{(1)}\otimes e_{i_2}^{(2)} \otimes \cdots \otimes
e_{i_{n-1}}^{(n-1)}, \qquad 1 \leq i_1 < i_2 < \cdots < i_{n-1} \leq N
\end{equation}
appearing in $Z_n^N$.  Let $G = C_2^{n-2}$ identified with subsets $S$ of
$[n-2]$, let $G^0 = \stab(\ui) \leq C_2^{n-2}$ be the subgroup consisting of $S$
for which $\utau_S \cdot \be = \be$, and let  $G^1 = C_2^{n-2}/G^0$.  By the 
group action property, for all $x \in G^0$, for all $y \in G^1$, $\utau_{x+y}\be 
= \utau_{y}\be$ so that when $G^0 \neq \{1\}$,  $\hat{\tau}\cdot \be = 
0$.

A necessary and sufficient condition for $G^0 = \{1\}$ is that, for some $1\leq
j \leq N'$,
\begin{align}
 &2^{n-2}(j-1)k < i_1 \leq 2^{n-2}(j-1)k + k\\
\notag \forall 1 < \ell \leq n-1 \qquad & 2^{n-2}(j-1)k + 2^{\ell-2}k < i_\ell \leq
2^{n-2}(j-1)k + 2^{\ell-1}k,
\end{align}
and $\tau_j = \one_{n-2} \in C_2^{n-2}$.  In words, the indices must all belong
to a common block of length $2^{n-2}k$ acted on by a single factor from $G_k$,
within this block, the first $2^{\ell-1}k$ elements must contain $i_\ell$ and
the second $2^{\ell-1}k$ must contain $i_{\ell+1}$ for $\ell = 1, 2, ..., n-2$,
and the factor $\tau_j$ acting on the block must be the element $\one_{n-2}$ of
the hypercube $C_2^{n-2}$.  

The product formula given summarizes this condition.

\end{proof}

\begin{proof}[Proof of Proposition \ref{char_fun_proposition}]
Assume without loss that $|\xi| \geq N^{- \frac{n-1}{2}}$.   Let $k = \max\left(\left\lfloor |\xi|^{-\frac{2}{n-1}}\right
\rfloor, M(\mu)\right)$, where $M(\mu)$ is a constant depending upon $\mu$.  
By the triangle inequality and H\"{o}lder's inequality,
\begin{align}
 \left| \hat{Z}_{n, \mu}^N(\xi)\right| & \leq \left|\E_{\bU_N} \left[ \chi_k(\xi, \uw)\right]\right|
 \\\notag & \leq \E_{\bU_N} \left[ \left| \chi_k(\xi, \uw)\right| \right]\\
 \notag & \leq \E_{\bU_N} \left[ \left|\chi_{k}(\xi, \uw)\right|^{2^{n-2}}\right]^{\frac{1}{2^{n-2}}}\\
 \notag & \leq \E_{\bU_N} \left[F_k(\xi,\uw) \right]^{\frac{1}{2^{n-2}}}.
\end{align}
Since disjoint blocks are i.i.d.,  Lemma \ref{factorization_lemma} implies that the expectation of $F_k(\xi, \uw)$ factors as a product
\begin{align}
 \E_{\bU_N} \left[F_k(\xi,\uw) \right] = \left(1 - \frac{1}{2^{n-2}} + \frac{\E\left[F_{k,1}(\xi,\uw)\right]}{2^{n-2}}\right)^{N'}. 
\end{align}
In the limit as $|\xi| \downarrow 0$, \begin{equation}\xi \sum_{\tau'
\in C_2^{n-2}}(-1)^{|\tau'|}Z_{n}^{2^{n-2}}((\tau + \tau') \cdot \uw_{j,k})\end{equation} has a continuous limiting distribution, which is a polynomial of degree $n-1$ in independent normal random variables, and hence the characteristic function $\E[F_{k, 1}(\xi, \uw)]$ has size bounded uniformly away from 1 for all $|\xi|$ smaller than a fixed constant $\epsilon(\mu)$.
It follows that
\begin{equation}
 \left| \hat{Z}_{n, \mu}^N(\xi)\right| \leq \exp(-C' N') \leq \exp\left(-C |\xi|^{\frac{2}{n-1}} N\right).
\end{equation}

To handle the remaining range $\epsilon(\mu) \leq |\xi| \leq \frac{1}{2}$, choose $N_1,
N_2, ..., N_{n-1}$ minimal such that $\mu^{*N_i}$ gives positive mass to the $i$th
standard basis vector in $\bR^{n-1}$.  Set $k = N_1 + ... + N_{n-1}$ and recall that $\mu$ assigns positive probability to 0.  Then with
$\mu^{\otimes 2^{n-2}k}$-positive probability, for each $1 \leq j \leq n-1$, $\omega_{2^{j-1}}$ is the $j$th standard basis vector and all other $\omega_i$ are 0.  For this configuration, $Z_n^{2^{n-2}}(\tau \cdot \uw_{1,k}) =1$ if $\tau$ is the identity, and 0 otherwise. Again, this gives that the characteristic function is uniformly bounded away from 1.  We thus conclude, as before, that
\begin{equation}
\left| \hat{Z}_{n, \mu}^N(\xi)\right|  \leq\left|\hat{Z}_{n, \mu}^N(\xi) \right| \leq  \exp(- C N).
\end{equation}

\end{proof}

\appendix

\section{The characteristic function of a Gaussian measure on the Heisenberg 
group}\label{char_fun_section}
This section proves Theorem \ref{char_fun_theorem}, which gives a
rate of convergence to the characteristic function of a Gaussian measure on the Heisenberg group when the steps in the walk are normally distributed in the abelianization.

Recall that
\begin{align}
 &I(\alpha,  \xi; N) = \int_{(\bR^{2})^N} 
e_{-\alpha}\left(\frac{\oux}{\sqrt{N}}\right) e_{-\xi}\left(\frac{ 
H^*\left(\ux\right)}{N}\right) d\nu_2^{\otimes N}\left(\ux\right)
\end{align}
where $\nu_2(x) = \frac{1}{2\pi } \exp\left(-\frac{\|x\|^2}{2} 
\right)$.

First consider the case $\alpha = 0$.
Integrate away $\ux^{(1)}$ to obtain,
\begin{align}
&I(0, \xi; N) = \frac{1}{(2\pi)^{\frac{N}{2}}}\int_{\bR^N}
\exp\left(- \frac{1}{2} \uy^t \left(\left(1-\xi_0^2\right) I_N + \xi_0^2
H\right)\uy\right) d\uy
\end{align}
where
\begin{align}
 \xi_0 = \frac{\pi \xi}{N}, \qquad & H_{i,j} = N - 2|i-j|;
\end{align}
this follows from $
 H^*(\ux) = \frac{1}{2} \sum_{i < j} (-1)^{\delta(j<i)}x_i^{(1)}x_j^{(2)},$ $\hat{\eta}(\xi) = e^{-2\pi^2 \xi^2}
$ for a standard one dimensional Gaussian, 
and
\begin{equation}\uy^t ( H-I_N)\uy=\sum_{i=1}^N 
\left(\sum_{j \neq i} (-1)^{\delta(j<i)} y_j \right)^2.\end{equation}
Thus,
\begin{equation}
 I(0,\xi;N) = \frac{1}{\sqrt{\det\left(\left(1-\xi_0^2\right) I_N + \xi_0^2 H 
\right)}},
\end{equation}
as may be seen by using an orthogonal matrix to diagonalize the quadratic form.

We perform elementary row operations to simplify the computation of the determinant.
Let
\begin{equation}
 U_- = I_N - \sum_{i=1}^{N-1} e_i \otimes e_{i+1}.
\end{equation}
Thus 
\begin{equation}
 \left(U_-^t U_-\right)_{i,j} = \left\{\begin{array}{lll}1, && i = j = 1\\ 2, && i = j > 1\\-1, && |i-j| = 1\\ 0, && \text{otherwise} \end{array}\right. 
\end{equation}
and 
\begin{equation}
 \left(U_-^t H U_- \right)_{i,j} = \left\{\begin{array}{lll}N, && i = j = 1\\ -2, && i = 1,\, j > 1 \text{ or } j=1,\, i > 1\\ 4, && i = j > 1 \\ 0, && \text{ otherwise} \end{array} \right.,
\end{equation}
so that
\begin{align}\label{prediagonal_matrix}
  &U_-^t\left(\left(1-\xi_0^2\right)I_N + \xi_0^2 H\right)U_-\\\notag& = (1 + 
\xi_0^2)\Biggl[2I_N - \frac{1-\xi_0^2}{1 + \xi_0^2} \sum_{i=1}^{N-1} (e_i 
\otimes e_{i+1} + e_{i+1}\otimes e_i)\\\notag& - \frac{2\xi_0^2}{1+\xi_0^2} 
\sum_{i=1}^N (e_1 \otimes e_i + e_i \otimes e_1) + \frac{(N+1)\xi_0^2 -1}{1 
+\xi_0^2} e_1 \otimes e_1 \Biggr].
\end{align}

Set $\zeta = \frac{1-\xi_0^2}{1 +\xi_0^2}$.  We diagonalize the tridiagonal matrix with 2's on the diagonal and $-\zeta$ on the first sub and super diagonal by working from the lower right corner and adding up and to the left, and treat the remainder of the matrix as a rank 2 perturbation.

Define sequences
\begin{align}
& \ve_1 = 2, \qquad \forall i \geq 1, \; \ve_{i+1} = 2 - \frac{\zeta^2}{\ve_i}\\ \notag
&\pi_0 = 1, \qquad \forall i \geq 1, \; \pi_i = \prod_{j=1}^i \ve_i\\ \notag
& \delta_1 = 1, \qquad \forall i \geq 1, \; \delta_{i+1} = 1 + \frac{\zeta 
\delta_i}{\ve_i}.
\end{align}
These parameters have the following behavior with proof postponed until the end
of this section.
\begin{lemma}\label{greek_eval_lemma}
For $\xi  \in \left(0 , N^{\frac{1}{2}}\right]$ the following asymptotics
hold
\begin{align}\notag
 \pi_N &= N \frac{\sinh(2\pi\xi)}{2\pi\xi} \left(1
+ O\left(\frac{1 + \xi^2}{N}\right)\right)\\
 \ve_N &= 1 + \frac{2\pi\xi}{N} \coth(2\pi\xi)  \left(1 +
O\left(\frac{1 + \xi^2}{N} \right)\right) \\\notag
 \delta_N &=\frac{N \tanh \pi\xi}{2\pi\xi} 
\left(1+ O\left(\frac{1 + \xi^2}{N} \right) \right) \\\notag
 \sum_{j=1}^{N-1} \frac{\delta_j^2}{\ve_j} &= 
\frac{N^3}{8\pi^3\xi^3}\left[2\pi\xi -
2\tanh \pi \xi \right]\left(1 + O\left(\frac{1 + \xi^2}{N}\right)\right).
\end{align}

\end{lemma}

Set
\begin{equation}
 L_\ve = I_N + \zeta \sum_{i=1}^{N-1} \frac{e_{i+1} \otimes e_i}{\ve_{N-i+1}},
\qquad D_\ve = \frac{1}{(1 + \xi_0^2)}\sum_{i=1}^N \frac{e_{i}\otimes
e_i}{\ve_{N+1-i}}.
\end{equation}
The diagonalization process is summarized in the following matrix equation, in which $D_\varepsilon^{\frac{1}{2}}$ is multiplied on left and right to obtain 1's on the diagonal 
\begin{align}
 &D_{\ve}^{\frac{1}{2}} L_\ve^t U_-^t\left(\left(1-\xi_0^2\right)I_N + \xi_0^2
H\right) U_- L_\ve D_{\ve}^{\frac{1}{2}}= I_N + P
\end{align}
and where $P$ is the rank two symmetric matrix which results from applying the diagonalization operators to the second line on the right hand side of (\ref{prediagonal_matrix}),
\begin{equation}
 P = \frac{-2\xi_0^2}{1 + \xi_0^2} \sum_{i=1}^N \frac{\delta_{N +
1-i}}{\sqrt{\ve_N \ve_{N-i+1}}}(e_1 \otimes e_i + e_i \otimes e_1) +
\frac{(N+1)\xi_0^2 -1}{\ve_N(1 + \xi_0^2)} e_1 \otimes e_1.
\end{equation}

Then, for some orthogonal matrix $O$, and $\lambda_{+} \geq \lambda_{-}$,
\begin{equation}O^t (I_N + P) O = (\lambda_+ e_1\otimes e_1 +\lambda_- e_2 \otimes e_2) \oplus
I_{N-2}.\end{equation}
By direct calculation, expanding by the top row, $\det (I_N + P)$ is equal to the $e_1 \otimes e_1$ entry plus the sum of the squares of the $e_1 \otimes e_i$ entries, $1 < i \leq N$,
\begin{align}
\label{determinant} \det\left(I_N + P \right) = \lambda_+ \lambda_- &= \left(1 
- \frac{1}{\ve_N(1 +
\xi_0^2)}\right) +\frac{(N+1)\xi_0^2}{\ve_N (1 + \xi_0^2)}\\\notag & \qquad-
\frac{4\xi_0^2}{1 + \xi_0^2} \frac{\delta_N}{\ve_N} - \frac{4\xi_0^4}{(1 +
\xi_0^2)^2 \ve_N} \sum_{j=1}^{N-1} \frac{\delta_i^2}{\ve_i}\\\notag
 &= \frac{\pi\xi \coth \pi \xi}{N} \left(1 + O\left(\frac{1 +
\xi^2}{N}\right) \right).
\end{align}
Since
\begin{equation}\label{volume}\det\left(D_\ve\right)^{-1} = (1 + \xi_0^2)^N
\pi_N =
\left(1 + O\left(\frac{1 + \xi^2}{N}\right)
\right)\frac{N\sinh 2\pi \xi}{2\pi \xi },\end{equation}
\begin{align}
\det\left(\left(1-\xi_0^2\right) I_N + \xi_0^2 H
\right)=&\left(\cosh \pi \xi\right)^2 \left(1 + O\left(\frac{1 + \xi^2}{N}
\right)\right).
\end{align}

Now consider the general case in which $\alpha \neq 0$.  Treat $\ux
$ as $N$ vectors in $\bR^2$.  When $\SO_2(\bR)$ acts diagonally on 
$(\bR^2)^N$ rotating each $x_i$ simultaneously, $H^*$ and the Gaussian density 
are preserved.  Thus, 
  $I(\alpha, \xi;N) = I((0,\|\alpha\|)^t , 
\xi;N)$.
Calculate
\begin{equation}
[1,1,\cdots, 1]U_- L_\varepsilon D_{\varepsilon}^{\frac{1}{2}} =
\frac{e_1}{\sqrt{\epsilon_N (1 + \xi_0^2)}}.
\end{equation}
It follows that after making the change of coordinates $\uy' =: U_1
L_{\varepsilon}D_{\varepsilon}^{\frac{1}{2}} \uy$ the phase has magnitude
$\frac{2\pi \|\alpha\|}{\sqrt{N\varepsilon_N(1 + \xi_0^2)}}$ and is now in the 
$e_1$ direction.
Let
$v_+, v_-$ be unit vectors generating the eigenspaces $\lambda_+, \lambda_-$
respectively. Since $e_1$ lies in the span of $v_+, v_-$ it follows
\begin{equation}
I(\alpha, \xi;N) = \exp \left(\frac{-2\pi^2\|\alpha\|^2}{N\varepsilon_N(1 + 
\xi_0^2)}
\left(\frac{\langle v_+, e_1 \rangle^2}{\lambda_+} + \frac{\langle v_-,
e_1\rangle^2}{\lambda_-} \right) \right)I( 0, \xi).
\end{equation}
Calculate
\begin{align}\label{trace}
T &= \lambda_+ + \lambda_- = 2 + e_1^t P e_1 = 1 + O\left(\frac{\xi^2}{N}
\right)
\end{align}
so that
\begin{equation}
(\lambda_+, \lambda_-) = \left(1 + O \left(\frac{1 + \xi^2}{N}  \right) \right)
\left(1, \frac{\pi \xi \coth \pi \xi}{N} \right).
\end{equation}
Also,
\begin{align}
 \langle v_+, e_1\rangle^2 + \langle v_-, e_1\rangle^2 &=1\\ \notag
 \lambda_+ \langle v_+, e_1\rangle^2 + \lambda_- \langle v_-, e_1\rangle^2 &= 1
+ e_1^t P e_1=
O\left(\frac{1 + \xi^2}{N} \right)
\end{align}
so that
\begin{equation}
\langle v_+, e_1\rangle^2 = O\left(\frac{1 + \xi^2}{N} \right),  \qquad  \langle
v_-, e_1\rangle^2 = 1 + O\left(\frac{1 + \xi^2}{N} \right).
\end{equation}
It follows that
\begin{equation}
\frac{\langle v_+, e_1 \rangle^2}{\lambda_+} + \frac{\langle v_-,
e_1\rangle^2}{\lambda_-} = \frac{N}{\pi \xi \coth \pi \xi} \left(1 
+ O\left(\frac{1
+ \xi^2}{N}\right)\right).
\end{equation}
In particular
\begin{equation}
I(\alpha, \xi;N) = \frac{\exp\left(\frac{-2\pi \|\alpha\|^2}{ \xi
\coth \pi \xi} \right)}{\cosh \pi \xi } \left(1 + O\left(\frac{(1 + 
\|\alpha\|^2)(1 +
\xi^2)}{N} \right) \right)
\end{equation}

\begin{proof}[Proof of Lemma \ref{greek_eval_lemma}]
 Recall $\zeta = \frac{1-\xi_0^2}{1+\xi_0^2}$.  $\pi_n$ satisfies the recurrence
 \begin{equation}
  \pi_n = 2 \pi_{n-1} - \zeta^2 \pi_{n-2}, \qquad \pi_0 = 1, \pi_1 = 2.
 \end{equation}
 The following closed forms hold,
  \begin{align}
  \pi_n &= \frac{(1 + \xi_0)^{2n+2} - (1-\xi_0)^{2n+2}}{4\xi_0 (1 +
\xi_0^2)^n}\\\notag
  \ve_n &= 1 + \frac{2\xi_0}{1 + \xi_0^2} \frac{(1 + \xi_0)^{2n} + 
(1-\xi_0)^{2n}}{(1 +
\xi_0)^{2n} - (1-\xi_0)^{2n}}\\\notag
\delta_n &= \frac{1}{2\xi_0}\frac{\left(\frac{1+\xi_0}{1-\xi_0}\right)^n +
\left(\frac{1-\xi_0}{1+\xi_0}\right)^n -
2}{\left(\frac{1+\xi_0}{1-\xi_0}\right)^n -
\left(\frac{1-\xi_0}{1+\xi_0}\right)^n} +
\frac{1}{2}.
 \end{align}

 The formula for $\pi_n$ is immediate from the recurrence relation, since
 \begin{equation}
  \frac{(1 + \xi_0)^2}{1 + \xi_0^2}, \qquad \frac{(1-\xi_0)^2}{1 + \xi_0^2}
 \end{equation}
are the two roots of $x^2 -2x + \zeta^2 = 0$.
 The formula for $\ve_n$ follows from $\ve_n = \frac{\pi_n}{\pi_{n-1}}$.  The
formula for $\delta_n$ is obtained on summing the geometric series
\begin{equation}
 \delta_n = \frac{\zeta^{n-1}}{\pi_{n-1}} \sum_{j=0}^{n-1} \frac{\pi_j}{\zeta^j},
\end{equation}
and use
\begin{align}
\frac{\pi_n}{\zeta^n} &= \frac{(1 + \xi_0)^{2n+2} -
(1-\xi_0)^{2n+2}}{4\xi_0(1-\xi_0^2)^n} \\\notag&=
\frac{1-\xi_0^2}{4\xi_0}\left[\left(\frac{1+\xi_0}{1-\xi_0} \right)^{n+1}-
\left(\frac{1-\xi_0}{1+\xi_0} \right)^{n+1} \right].
\end{align}

The claimed asymptotics for $\pi, \ve, \delta$ are straightforward.  For instance, to obtain the correct relative error in $\pi_n$, write
\begin{equation}
 \frac{\left(1 + \xi_0 \right)^{2n+2} - \left(1-\xi_0 \right)^{2n+2}}{4 \xi_0} = \frac{\left(1 + \xi_0 \right)^{2n+2} - \left(1-\xi_0 \right)^{2n+2}}{2\left(\left(1 + \xi_0 \right) - \left(1-\xi_0 \right)\right)}
\end{equation}
as a geometric series of positive terms.  In each term of the series, approximate the power with an exponential with acceptable relative error, then sum the sequence of exponentials.  The correct relative error may be obtained in the other cases similarly.

Using the exact formulae for $\ve$ and $\delta$ yields
\begin{align}
 \frac{\delta_j^2}{\ve_j}
 &= \left(1 + O\left(\frac{1}{j} + \frac{\xi^2}{N}\right)\right)
\frac{N^2}{4\pi^2\xi^2} \tanh \left(\frac{j\pi\xi}{N}\right)^2.
\end{align}
Approximating with a Riemann sum,
\begin{equation}
 \sum_{j=1}^{N-1} \frac{\delta_j^2}{\ve_j} = \left(1 +
O\left(\frac{1+\xi^2}{N}\right)\right) \frac{N^3}{4\pi^2\xi^2} \int_0^1
\tanh\left(t \pi\xi\right)^2 dt
\end{equation}
which gives the claimed estimate.
\end{proof}

\bibliographystyle{plain}

\end{document}